\documentclass[11pt]{article}

\usepackage{amsmath}
\usepackage{amsfonts}
\usepackage{amssymb}
\usepackage{wasysym}
\usepackage{amsthm}

\usepackage[margin=1in]{geometry}

    \theoremstyle{definition}\newtheorem{rema}{Remark}[section]
    \theoremstyle{plain}\newtheorem{propo}[rema]{Proposition}
    \newtheorem{theo}[rema]{Theorem}

    \newtheorem{proposition}[rema]{Proposition}
    \theoremstyle{definition}

\newcommand{\ignore}[1]{}
\numberwithin{equation}{section}
\newcommand{\gga}{G\"ollnitz-Gordon-Andrews }

\begin{document}
\title{A motivated proof of the G\"{o}llnitz-Gordon-Andrews 
identities}
\author{Bud Coulson, Shashank Kanade, James Lepowsky, Robert McRae, 
\\ Fei Qi, Matthew C. Russell and Christopher Sadowski}
\date{}
\bibliographystyle{alpha}
\maketitle

\begin{abstract}
We present what we call a ``motivated proof'' of the \gga identities.
A similar motivated proof of the Rogers-Ramanujan
identities was previously given by G.\ E.\ Andrews and R.\ J.\ Baxter, 
and was subsequently generalized to Gordon's identities
by J.\ Lepowsky and M.\ Zhu. We anticipate that the present proof
of the \gga identities will illuminate certain twisted 
vertex-algebraic constructions.
\end{abstract}

\section{Introduction}

The \gga identities form a family of partition identities 
somewhat analogous to the 
Gordon-Andrews generalizations of the Rogers-Ramanujan identities. 
As presented in Chapter 7 of \cite{A3} (which also gives
the Rogers-Ramanujan and Gordon-Andrews identities), 
these identities state that 
for any $k\ge 2$ (for $k=1$, one gets the trivial identity $1=1$) 
and $i=1,\dots,k$,
$$\prod_{\substack{m\geq 1,\,\,m\,\not\equiv\, 2\,(\mathrm{mod}\,4), \\ 
m\,\not\equiv \, 0,\: 2k\pm(2i-1)\,(\mathrm{mod}\,4k)}} \frac{1}{1-q^m} 
= \sum\limits_{n\ge 0} d_{k,i}(n)q^n, $$
where 
$d_{k,i}(n)$ is the number of partitions $(b_1,\dots,b_s)$ of $n$ (with $b_p\geq 
b_{p+1}$), satisfying 
the following conditions:
\begin{enumerate}
\item No odd parts are repeated,
\item $b_p-b_{p+k-1}\geq 2$ if $b_p$ is odd,
\item $b_p-b_{p+k-1} >2$ if $b_p$ is even, and
\item at most $k-i$ parts are equal to $1$ or $2$.
\end{enumerate}
Here we have replaced $i$ by $k-i$ in the statement of these identities
in Theorem 7.11 of \cite{A3}.  Also, here and below, $q$ is a formal
variable.

For $k=2$, these identities were independently discovered by H.\ G\"ollnitz 
\cite{Gol} 
and B.\ Gordon \cite{G2},
and were subsequently generalized to all $k$ by G.\ E.\ Andrews \cite{A2}. 
As noted in \cite{SW}, L. J. Slater had published in \cite{S} analytic 
counterparts of these identities 
even earlier than \cite{Gol} and \cite{G2}, and 
it has recently been pointed out by A. Sills that two analytic identities 
equivalent to the analytic G\"{o}llnitz-Gordon identities were 
actually recorded by Ramanujan in his ``Lost Notebook'' long 
before Slater rediscovered them (cf. \cite{ABe}, page 37).

In this paper, we shall present what we call a ``motivated proof'' of
these identities. 
(This designation has become a technical term in our work,
and we shall omit the quotation marks.)
To set the stage, we begin by explaining the
motivated proof of the classical Rogers-Ramanujan identities
carried out by Andrews and R. J. Baxter in \cite{AB}.  Recall the
Rogers-Ramanujan identities:
\begin{eqnarray*}
\prod_{m \geq 1, \,\, m \not\equiv 0, \, \pm 2 \, (\text{mod } 5)}
\frac{1}{1-q^m}
& = &
\sum_{n \geq 0} p_1 (n)q^n
\end{eqnarray*}
and
\begin{eqnarray*}
\prod_{m \geq 1, \,\, m \not\equiv 0, \, \pm 1 \, (\text{mod } 5)}
\frac{1}{1-q^m}
& = &
\sum_{n \geq 0} p_2 (n)q^n,
\end{eqnarray*}
where
\begin{eqnarray*}
p_1 (n)
& = & 
\text{the number of partitions of $n$ for which adjacent parts have}
\nonumber \\
& &
\quad \text{difference at least 2}
\end{eqnarray*}
and
\begin{eqnarray*}
p_2 (n)
& = & 
\text{the number of partitions of $n$ for which adjacent parts have}
\nonumber \\
& & 
\quad \text{difference at least 2 and in which 1 does not appear.}
\end{eqnarray*}
Note that the first product is the generating function for the number of partitions
with parts congruent to $\pm 1\pmod 5$, 
and the second product gives the number of partitions with parts congruent to $\pm 2\pmod 5.$

Upon examining the two sum sides 
and assuming the truth of the Rogers-Ramanujan identities, one 
sees that if one subtracts the second Rogers-Ramanujan product, called $G_2(q)$ in \cite{AB},
from the first, called $G_1(q)$, then one gets a power series with nonnegative coefficients.
An explanation of this phenomenon using only the product sides
was asked for by Leon Ehrenpreis. While answering this 
question, 
Andrews and Baxter
were naturally led to a proof of the Rogers-Ramanujan identities themselves. 
Their idea was to form a sequence of power series as follows: 
Let $G_3(q) = (G_1(q) - G_2(q))/q$, $G_4(q) = (G_2(q) - G_3(q))/q^2$, and more 
generally,
$G_i(q) = (G_{i-2}(q)-G_{i-1}(q))/q^{i-2}$. 
The key point was that one then observes {\em empirically} that 
for all $i$, $G_i(q)$ is a formal power series with constant term  $1$ and 
that $G_i(q)-1$
is divisible by $q^i$. This ``Empirical Hypothesis'' was then proved using only 
the product sides
of the Rogers-Ramanujan identities, and the 
Empirical Hypothesis 
immediately led to a proof of the identities themselves. 
This proof was in fact closely related to Baxter's proof in \cite{B} 
and also to Rogers's and Ramanujan's proof in \cite{RR}.
The truth of the Empirical Hypothesis
also gave a (motivated) answer to Ehrenpreis's question as well, by means
of a variant of the same argument.

In \cite{LZ}, the Andrews-Baxter motivated proof of the Rogers-Ramanujan 
identities
was generalized to the setting of Gordon's family of identities
\cite{G1} (see also Theorem 7.5 of \cite{A3}),
and in the present paper, we provide a motivated proof of the \gga identities 
analogous
to the motivated proof in \cite{LZ}. A certain ``shelf'' picture, implicit in 
\cite{AB},
became transparent in \cite{LZ}. Moreover, in \cite{LZ} the appropriate
Empirical Hypothesis was not actually observed empirically, but was instead
proved directly from the product sides and this in turn was used to prove 
Gordon's identities. 
This will be the case for the present paper as well.
Therefore, as in \cite{LZ}, we shall use the term ``Empirical Hypothesis'' as a 
technical term.
As we discuss below, this Empirical Hypothesis approach to discovering and 
proving
new identities is expected to be valuable when the ``sum sides'' are not known.

The ``shelf'' picture, as clarified in \cite{LZ}, amounts to the following:
Imagine that the given product sides form a ``$0^{\text{th}}$ shelf'' of 
formal power series in $q$.
Now one constructs a ``$1^{\text{st}}$ shelf'' by taking certain judiciously chosen 
linear 
combinations,
over the field of fractions of the ring of polynomials in the formal variable 
$q$,
of the power series on the $0^{\text{th}}$ shelf. 
Using analogous linear combinations, one repeats this process recursively,
to  successively build higher shelves of formal power series. 
At every step, one ensures that the power series appearing on the 
$j^{\text{th}}$ shelf
are of the form $1+q^{j+1}f(q)$, where $f(q)\in\mathbb{Z}[[q]]$. 
This is essentially the Empirical Hypothesis of Andrews-Baxter and of Lepowsky-Zhu. 

For the present context of the \gga identities, and actually for Gordon's 
identities as well, 
one way of discovering the linear combinations is to 
use linear combinations predicted by the 
$(a,x,q)$-recursions in Lemma 7.2 of \cite{A3}, with $a$ suitably specialized
as in \cite{A3}, but now with $x$ specialized to successively higher powers of $q$,
to construct the higher shelves 
(although this was not the method used in \cite{LZ} for discovering the appropriate
linear combinations). We elaborate on this process in Appendices A and C below. 
This issue  was already pointed out in the context of the Rogers-Ramanujan 
identities in \cite{AB},
and it is one more reason why we use ``Empirical Hypothesis'' as a purely technical term.
For the \gga identities and Gordon's identities, these linear combinations
are ``easy,'' in the sense that each power series on the $j^{\text{th}}$ shelf 
is obtained by multiplying certain power series on the $(j-1)^{\text{st}}$ shelf 
by certain Laurent polynomials and then adding.
In fact, for the \gga identities, these linear combinations can be simplified 
further if one allows
recursion on the $j^{\text{th}}$ shelf as well.
In retrospect, therefore, it might not have been too difficult to 
guess the choice of these linear combinations 
using pure experimentation with the aim of producing a sound Empirical 
Hypothesis.

As is well known, partition identities of Rogers-Ramanujan type are intimately related
to the representation theory of vertex operator algebras, 
and it turns out that the philosophy
of ``motivated proofs'' fits extremely well with vertex-operator-theoretic 
investigations of such identities. 
As is explained in \cite{LZ}, one starts with the product
sides of such identities as ``given,'' as was the case in \cite{LM}, and then the problem
is to exhibit and explain the corresponding combinatorial sum sides
in such a way that the relevant numbers of partitions arise as the dimensions 
of vector spaces constructed from natural generalized vertex operators.
In early vertex operator theory, 
which was indeed motivated by this very problem,
the combinatorial sum sides were built using 
monomials in principally twisted $Z$-operators, 
as developed in \cite{LW2}-\cite{LW4} using \cite{LW1} as a starting 
point. 
In fact, the $Z$-algebraic structure developed in \cite{LW2}-\cite{LW4} was used 
to give
a vertex-operator-theoretic interpretation of a certain family of 
Gordon-Andrews-Bressoud 
identities, as well as a vertex-operator-theoretic proof of the Rogers-Ramanujan 
identities, in the context of the affine Lie algebra $A_1^{(1)}$. 
A related vertex-operator-theoretic proof of the Gordon-Andrews-Bressoud 
identities was given in \cite{MP}. 
Untwisted analogues of the $Z$-operators were developed and exploited in 
\cite{LP}.
See \cite{L} for a number of related developments.
All of these algebraic structures came to be understood in a general 
vertex-algebraic framework based on generalized vertex operator
algebras, modules, twisted modules and intertwining operators;
see for instance \cite{DL}, \cite{H}, 
\cite{CLM1}, \cite{CLM2}, \cite{Cal1}, \cite{Cal2}, 
\cite{CalLM1}--\cite{CalLM4} and \cite{Sa}.
However, in the setting of the principally twisted $Z$-algebras, a 
natural mechanism for
interpreting the formal variable $x$ (referred to above), which counts the 
number of parts, is yet to be found. 
Therefore, proofs involving only the formal variable $q$,
which is concerned with the integer being partitioned, are desirable. 
After such  proofs are found, one can conceive of
interpreting the steps in the proofs 
by means of twisted intertwining operators among modules for
generalized vertex operator algebras; this is related to an early idea
of J. Lepowsky and A. Milas.  The desire to better understand 
the connection between partition identities and vertex operator algebra theory 
is the main incentive
for seeking ``motivated proofs'' as in the present work.
See the Introduction in \cite{LZ} for further discussion of these issues.

An important feature of such motivated proofs 
is that they do not require knowledge of
Andrews's  generalization of Rogers's (somewhat mysterious)
expression, formula (7.2.2) in \cite{A3}, which
after appropriate specializations of $a$ and $x$ and the 
use of the Jacobi triple product
identity gives the corresponding product sides of the relevant identities, 
and which on the other hand leads to recursions (Lemma 7.2 in \cite{A3})
whose solutions yield the sum sides, as we recalled above.
This important feature will be useful for the investigation of more 
complicated identities arising from the theory of
vertex operator algebras. 
In fact, a number of additional works on motivated proofs are underway.

It was in his work on statistical mechanics that Baxter independently
discovered the Rogers-Ramanujan identities \cite{B}.  There is a vast
literature on connections between Rogers-Ramanujan-type identities and
statistical mechanics, including \cite{ABF} as well as our motivating
paper \cite{AB} and work of A. Berkovich, B. McCoy, A. Schilling,
S. O. Warnaar and many others; we refer the reader to the references
in \cite{GOW} for a large number of relevant works, including
connections between Rogers-Ramanujan-type identities and a variety of
other fields.  In \cite{GOW}, M. Griffin, K. Ono and Warnaar have
given a framework extending the Rogers-Ramanujan identities,
incorporating the Hall-Littlewood polynomials, in order to deduce new
arithmetic properties of certain $q$-series.

This paper is organized as follows: 
In Section \ref{sec:shelf0} we recall the use of the Jacobi triple product 
identity to re-express the product sides of the \gga identities. 
The resulting formal power series form our $0^{\text{th}}$ shelf.
In Section \ref{sec:closed-form} we determine closed-form expressions 
for our higher shelves.
We use these closed-form expressions to formulate and prove our 
Empirical Hypothesis in Section \ref{sec:EH}. 
We recast the relevant recursions in an elegant
matrix formulation and prove one form of our main theorem in Section 
\ref{sec:matrixinterp}. 
We complete the motivated proof in Section \ref{sec:combinterp}.
In Appendix \ref{app:xqforGGA}, we compare our method with
the proof in \cite{A3}. 
Finally, the developments in the present work suggested to us
some natural enhancements of \cite{LZ}, which we proceed to give
in Appendices \ref{app:LZremarks} and \ref{app:xqforGRR}.

Just as in \cite{LZ}, throughout this paper we treat power series as 
purely formal series rather than as convergent series (in suitable domains) 
in complex variables.

This paper is an outcome of J. Lepowsky's Spring 2014 Rutgers graduate
course, ``The theory of partitions and vertex operator algebras.''

\section{The formal series $G_l(q)$}
\label{sec:shelf0}
Throughout this paper, we fix $k\geq 2$. The main tool we will use for the 
motivated proof is an infinite sequence $G_l(q)$, $l \ge 1$, of formal power 
series that we shall generate recursively from the product sides of the 
G\"{o}llnitz-Gordon-Andrews identities. We will want to arrange these formal 
series in ``shelves'' as follows: for each $j\geq 0$, the $j^{\text{th}}$ shelf 
will 
consist of the $k$ formal series 
\begin{equation*}
 G_{(k-1)j+i}(q)
\end{equation*}
for $1\leq i\leq k$. Note that the shelves will overlap, since
\begin{equation}\label{eq:edgematchingalgebra}
 (k-1)j+k=(k-1)(j+1)+1,
\end{equation}
that is, the first series in the $j+1^{\text{st}}$ shelf is defined to be the 
last 
series of the $j^{\text{th}}$ shelf.

We start with the case $j=0$ by defining the series $G_i(q)$ for $1\leq i\leq 
k$ 
as the product sides of the G\"{o}llnitz-Gordon-Andrews identities as presented 
in Theorem 7.11 of \cite{A3}: 
\begin{equation}\label{Gidef}
G_i(q)=\prod_{\substack{m\geq 1,\,\,m\,\not\equiv\, 2\,(\mathrm{mod}\,4), \\ 
m\,\not\equiv \, 0,\: 2k\pm(2i-1)\,(\mathrm{mod}\,4k)}} (1-q^m)^{-1}.
\end{equation}
From the sequence $G_i(q)$ of formal series, we shall recover
the G\"{o}llnitz-Gordon-Andrews identities. For $k=2$, $i=1,2$, these are the 
G\"{ollnitz}-Gordon identities, and the most general form of these identities, 
that is, for all $k\geq 2$, is due to Andrews.

\begin{rema}\label{prodsideinterp}
 Observe that for $1\leq i\leq k$, $G_i(q)$ is the generating function for 
partitions into parts not congruent to $2$ (mod $4$) and not congruent to $0, 
2k\pm(2i-1)$ (mod $4k$).
\end{rema}

Now for $j\geq 1$, we define the series $G_{(k-1)j+i}(q)$, $1\leq i\leq k$, 
recursively. As mentioned in (\ref{eq:edgematchingalgebra}),
\begin{equation*}
 G_{(k-1)j+1}(q)=G_{(k-1)(j-1)+k}(q)
\end{equation*}
tautologically, and for $2\leq i\leq k$, we define
\begin{equation}\label{recursion}
G_{(k-1)j+i}(q)=\dfrac{G_{(k-1)(j-1)+k-i+1}(q)-G_{(k-1)(j-1)+k-i+2}(q)}{q^{
2j(i-1)}}-q^{-1} G_{(k-1)j+i-1}(q).
\end{equation}
\begin{rema}
The recursions \eqref{recursion} can be predicted by Lemma 7.2 in \cite{A3}, as 
mentioned in the Introduction. Alternatively, they could be derived from 
knowledge of the sum sides of the G\"{o}llnitz-Gordon-Andrews identities;
see Sections \ref{sec:matrixinterp} and 
\ref{sec:combinterp} below.
\end{rema}

The product definition \eqref{Gidef} of the $0^{\text{th}}$-shelf series 
$G_i(q)$ for 
$1\leq i\leq k$ is not useful for studying the series in higher shelves, so we 
will need to rewrite the $G_i(q)$ using the Jacobi triple product identity, and 
work with the resulting alternating sums. The Jacobi triple product identity 
states (see for example Theorem 2.8 in \cite{A3}): 
\begin{equation*}
\sum_{n\in\mathbb{Z}} (-1)^n z^n q^{n^2}=\prod_{m\geq 0} (1-q^{2m+2})(1-z 
q^{2m+1})(1-z^{-1} q^{2m+1}).
\end{equation*}
We rewrite the sum side:
\begin{align*}
\sum_{n\in\mathbb{Z}} (-1)^n z^n q^{n^2}=\sum_{n\geq 0} (-1)^n (z^n 
q^{n^2}-z^{-n-1} q^{(n+1)^2})=\sum_{n\geq 0} (-1)^n z^n q^{n^2}(1-z^{-2n-1} 
q^{2n+1}).
\end{align*}
Specializing $q\mapsto q^{2k}$ and then $z\mapsto q^{2i-1}$, we have:
\begin{align}
& \prod_{m\geq 0} (1  -q^{4k(m+1)})(1-q^{2k(2m+1)+2i-1})(1-q^{2k(2m+1)-2i+1}) 
\nonumber \\
& \quad = \prod_{\substack{m\geq 1,\\ m\,\equiv\, 0,\: 
2k\pm(2i-1)\,(\mathrm{mod}\,4k)}}
\hspace{-0.5in}(1-q^m)\nonumber\\
& \quad = \sum_{n\geq 0} (-1)^n q^{2kn^2+(2i-1)n}(1-q^{(2n+1)(2k-2i+1)}) 
\nonumber \\
& \quad = \sum_{n\geq 0} (-1)^n 
q^{4k\binom{n}{2}+(2k+2i-1)n}(1-q^{(2k-2i+1)(2n+1)})\label{beforeF}.
\end{align}

We shall use the notation
\begin{equation} \label{eq:Fq}
F(q)=\prod_{m\,\not\equiv\, 2\,(\mathrm{mod}\, 
4)} (1-q^m)
\end{equation}
here and below.
Dividing the product and sum sides of \eqref{beforeF} by $F(q)$, we obtain
\begin{align}\label{eq:shelf0Gi}
 G_i(q) & =\dfrac{1}{F(q)}\sum_{n\geq 0} (-1)^n 
q^{4k\binom{n}{2}+(2k+2i-1)n}(1-q^{(2k-2i+1)(2n+1)}).
\end{align}
It is this form of the series $G_i(q)$ that we will always work with,
instead of their product expressions.

\section{Closed-form determination of the $G_l(q)$}\label{sec:closed-form}
\allowdisplaybreaks

This section contains the technical heart of the motivated proof, which is 
determining a closed-form expression for all the $G_l(q)$, $l \ge 1$, in terms 
of certain alternating sums. The following theorem is analogous to Theorem 2.1 
in \cite{LZ}, but the closed-form expressions are necessarily more complicated 
and the proof is more involved. As in \cite{LZ}, the proof is by induction on 
the shelf.
We have:
\begin{theo}\label{thm:closedform}
For any $j\geq 0$ and $1\leq i\leq k$,
\begin{align}\label{altsum}
G_{(k-1)j+i}(q)= & \,\dfrac{1}{F(q)} \sum_{n\geq 0}\dfrac{(-1)^n 
q^{4k\binom{n}{2}+(2k(j+1)+2i-1)n}(1-q^{2(n+1)})\cdots 
(1-q^{2(n+j)})}{(1+q^{2n+1})(1+q^{2(n+1)+1})\cdots 
(1+q^{2(n+j)+1})}\nonumber\\
& 
\hspace{4em}\cdot\left(1-q^{2(k-i+1)(2n+j+1)}+q^{2(n+j)+1}(1-q^{2(k-i)(2n+j+1)}
)\right).
\end{align}
\end{theo}

\begin{proof}
The expression (\ref{altsum}) gives two different fomulas for the ``edge'' cases
--- $i=1$ for $j \ge 1$ and $i=k$ for $j-1$ ---
as shown in~\eqref{eq:edgematchingalgebra},
and we 
prove first that they are compatible.
That is, we show that for any $j\geq 1$, the two formulas given for 
$$
G_{(k-1)(j-1)+k}(q)=G_{(k-1)j+1}(q)
$$ agree. 
Using $\widetilde{G}_{j,i}(q)$ to denote the right-hand side of 
(\ref{altsum}), 
we show that 
\begin{align} \label{edge-matching}
\widetilde{G}_{j,k}(q)=\widetilde{G}_{j+1,1}(q)
\end{align} 
for any $j\geq 0$.
We will call this equality ``edge-matching.'' We have:
 \begin{align*}
 &\widetilde{G}_{j,k}(q) =\lefteqn{\dfrac{1}{F(q)}
 \sum_{n\geq 0}\dfrac{(-1)^n 
q^{4k\binom{n}{2}+(2k(j+2)-1)n}(1-q^{2(n+1)})\cdots 
(1-q^{2(n+j)})}{(1+q^{2n+1})(1+q^{2(n+1)+1})\cdots 
(1+q^{2(n+j)+1})}(1-q^{2(2n+j+1{})})}\nonumber\\
 & =  \dfrac{1}{F(q)}\sum_{n\geq 0}\dfrac{(-1)^n 
q^{4k\binom{n}{2}+(2k(j+2)+1)n}(1-q^{2(n+1)})\cdots 
(1-q^{2(n+j)})}{(1+q^{2n+1})(1+q^{2(n+1)+1})\cdots (1+q^{2(n+j)+1})} 
q^{-2n}(1-q^{2(2n+j+1{})})\nonumber\\
 & =  \dfrac{1}{F(q)}\sum_{n\geq 0}\dfrac{(-1)^n 
q^{4k\binom{n}{2}+(2k(j+2)+1)n}(1-q^{2(n+1)})\cdots 
(1-q^{2(n+j)})}{(1+q^{2n+1})(1+q^{2(n+1)+1})\cdots (1+q^{2(n+j)+1})}( 
q^{-2n}(1-q^{2n})+(1-q^{2(n+j+1{})})) \nonumber\\
 & =  \dfrac{1}{F(q)}\sum_{n\geq 1}\dfrac{(-1)^n 
q^{4k\binom{n}{2}+(2k(j+2)+1)n}(1-q^{2n})(1-q^{2(n+1)})\cdots 
(1-q^{2(n+j)})}{(1+q^{2n+1})(1+q^{2(n+1)+1})\cdots (1+q^{2(n+j)+1})} 
q^{-2n}\nonumber\\
 &  \hphantom{=}\,\, + \mbox{} \dfrac{1}{F(q)}\sum_{n\geq 0}\dfrac{(-1)^n 
q^{4k\binom{n}{2}+(2k(j+2)+1)n}(1-q^{2(n+1)})\cdots 
(1-q^{2(n+j)})(1-q^{2(n+j+1{})})}{(1+q^{2n+1})(1+q^{2(n+1)+1})\cdots 
(1+q^{2(n+j)+1})}\nonumber\\
 & =   -\dfrac{1}{F(q)}\sum_{n\geq 0}\dfrac{(-1)^n 
q^{4k\binom{n}{2}+4kn+(2k(j+2)+1)(n+1)}(1-q^{2(n+1)})(1-q^{2(n+2)})\cdots 
(1-q^{2(n+j+1{})})}{(1+q^{2(n+1)+1})(1+q^{2(n+2)+1})\cdots 
(1+q^{2(n+j+1{})+1})} 
q^{-2(n+1)}\nonumber\\
 & \hphantom{=} + \dfrac{1}{F(q)}\sum_{n\geq 0}\dfrac{(-1)^n 
q^{4k\binom{n}{2}+(2k(j+2)+1)n}(1-q^{2(n+1)})\cdots 
(1-q^{2(n+j)})(1-q^{2(n+j+1{})})}{(1+q^{2n+1})(1+q^{2(n+1)+1})\cdots 
(1+q^{2(n+j)+1})}\nonumber\\
 & =  \dfrac{1}{F(q)}\sum_{n\geq 0}\dfrac{(-1)^n 
q^{4k\binom{n}{2}+(2k(j+2)+1)n}(1-q^{2(n+1)})\cdots 
(1-q^{2(n+j)})(1-q^{2(n+j+1{})})}{(1+q^{2n+1})(1+q^{2(n+1)+1})\cdots 
(1+q^{2(n+j+1{})+1})}\nonumber\\
 & \hspace{6em} 
\cdot(-q^{4kn+2k(j+2)+1-2(n+1)}(1+q^{2n+1})+(1+q^{2(n+j+1{})+1}))\nonumber\\
  & =  \dfrac{1}{F(q)}\sum_{n\geq 0}\dfrac{(-1)^n 
q^{4k\binom{n}{2}+(2k(j+2)+1)n}(1-q^{2(n+1)})\cdots 
(1-q^{2(n+j)})(1-q^{2(n+j+1{})})}{(1+q^{2n+1})(1+q^{2(n+1)+1})\cdots 
(1+q^{2(n+j+1{})+1})}\nonumber\\
  & \hspace{6em} 
\cdot(1-q^{2k(2n+j+2)}+q^{2(n+j+1{})+1}(1-q^{2k(2n+j+2)-2(2n+1)-2(j+1{})}
))\nonumber\\
  & =  \dfrac{1}{F(q)}\sum_{n\geq 0}\dfrac{(-1)^n 
q^{4k\binom{n}{2}+(2k(j+2)+1)n}(1-q^{2(n+1)})\cdots 
(1-q^{2(n+j)})(1-q^{2(n+j+1{})})}{(1+q^{2n+1})(1+q^{2(n+1)+1})\cdots 
(1+q^{2(n+j+1{})+1})}\nonumber\\
  & \hspace{6em} \cdot(1-q^{2k(2n+j+2)}+q^{2(n+j+1{})+1}(1-q^{2(k-1)(2n+j+2)})),
 \end{align*}
which is $\widetilde{G}_{j+1,1}(q)$, as desired.
 
We now prove (\ref{altsum}) by induction on $j$;
we shall need the ``edge-matching'' formula (\ref{edge-matching}) in this 
argument. For the case $j=0$ we multiply the numerator and denominator of 
\eqref{eq:shelf0Gi} by $1+q^{2n+1}$ to obtain
\begin{align*}
 G_i(q) & =\dfrac{1}{F(q)}\sum_{n\geq 0}\dfrac{(-1)^n 
q^{4k\binom{n}{2}+(2k+2i-1)n}(1+q^{2n+1})(1-q^{(2k-2i+1)(2n+1)})}{1+q^{2n+1}}
\nonumber\\
 & =\dfrac{1}{F(q)} \sum_{n\geq 0}\dfrac{(-1)^n 
q^{4k\binom{n}{2}+(2k+2i-1)n}(1-q^{2(k-i+1)(2n+1)}+q^{2n+1}(1-q^{2(k-i)(2n+1)}))
}{1+q^{2n+1}}
\end{align*}
for $1\leq i\leq k$, as desired.

Take $j\ge 0$. Assume that~\eqref{altsum} holds for all $i=1,\dots,k$.
We shall show that it holds for $j+1$ and all $i=1,\dots,k$ using induction on 
$i$.
The case $i=1$ holds by ``edge-matching'' since $G_{(k-1)(j+1)+1}=G_{(k-1)j+k}$.

Now assume that (\ref{altsum}) holds for $j+1$ and $i-1$ where $i=2,\dots,k-1$. 
By definition and the inductive hypothesis on $j$,
\begin{align}\label{inductioncalc}
& G_{(k-1)(j+1)+i}(q)  +q^{-1} G_{(k-1)(j+1)+i-1}(q) = 
\dfrac{G_{(k-1)j+k-i+1}(q)-G_{(k-1)j+k-i+2}(q)}{q^{2(j+1)(i-1)}}\nonumber\\
& \, = \dfrac{1}{q^{2(j+1)(i-1)}F(q)}\sum_{n\geq 0}\dfrac{(-1)^n 
q^{4k\binom{n}{2}+(2k(j+2)+1)n}(1-q^{2(n+1)})\cdots 
(1-q^{2(n+j)})}{(1+q^{2n+1})(1+q^{2(n+1)+1})\cdots 
(1+q^{2(n+j)+1})}\nonumber\\
&\hspace{11em}\cdot\left(q^{-2in}(1-q^{2i(2n+j+1)}+q^{2(n+j)+1}(1-q^{
2(i-1)(2n+j+1)}
)\right.\nonumber\\
 &\hspace{11em} \left. 
-q^{-2(i-1)n}(1-q^{2(i-1)(2n+j+1)}+q^{2(n+j)+1}(1-q^{2(i-2)(2n+j+1)})\right).
\end{align}
The term in parentheses in the last two lines of (\ref{inductioncalc}) may be 
rewritten as
\begin{align*}
q^{-2in}&(1-q^{2i(2n+j+1)})-q^{-2(i-1)n}(1-q^{2(i-1)(2n+j+1)})\nonumber\\
&  + 
q^{2j+1}(q^{-2(i-1)n}(1-q^{2(i-1)(2n+j+1)})-q^{-2(i-2)n}(1-q^{2(i-2)(2n+j+1)}
))\nonumber\\
= \mbox{} & q^{-2in}(1-q^{2n})+q^{2(i-1)(n+j+1)}(1-q^{2(n+j+1)})\nonumber\\
& + q^{2j+1}(q^{-2(i-1)n}(1-q^{2n})+q^{2(i-2)(n+j+1)}(1-q^{2(n+j+1)})),
\end{align*}
so we obtain
\begin{align}\label{twosums}
& G_{(k-1)(j+1)+i}(q)  +q^{-1} G_{(k-1)(j+1)+i-1}(q)\nonumber\\
& = \dfrac{1}{q^{2(j+1)(i-1)}F(q) }\sum_{n\geq 0}\dfrac{(-1)^n 
q^{4k\binom{n}{2}+(2k(j+2)+1)n}(1-q^{2(n+1)})\cdots 
(1-q^{2(n+j)})}{(1+q^{2n+1})(1+q^{2(n+1)+1})\cdots 
(1+q^{2(n+j)+1})}\nonumber\\
& \hphantom{=} 
\cdot\left(q^{-2in}(1-q^{2n})+q^{2(i-1)(n+j+1)}(1-q^{2(n+j+1)}
)\right)\nonumber\\
& + \dfrac{q^{2j+1}}{q^{2(j+1)(i-1)}F(q)}\sum_{n\geq 0}\dfrac{(-1)^n 
q^{4k\binom{n}{2}+(2k(j+2)+1)n}(1-q^{2(n+1)})\cdots 
(1-q^{2(n+j)})}{(1+q^{2n+1})(1+q^{2(n+1)+1})\cdots 
(1+q^{2(n+j)+1})}\nonumber\\
& \hphantom{=}
\cdot\left(q^{-2(i-1)n}(1-q^{2n})+q^{2(i-2)(n+j+1)}(1-q^{2(n+j+1)})\right).
\end{align}

We now analyze the first sum on the right-hand side of~\eqref{twosums}:
\begin{align}\label{inductioncalc2}
& \sum_{n\geq 0}\dfrac{(-1)^n 
q^{4k\binom{n}{2}+(2k(j+2)+1)n}(1-q^{2(n+1)})\cdots 
(1-q^{2(n+j)})}{(1+q^{2n+1})(1+q^{2(n+1)+1})\cdots 
(1+q^{2(n+j)+1})}\nonumber\\
&\hspace{2em}\cdot\left(q^{-2in}(1-q^{2n})+q^{2(i-1)(n+j+1)}(1-q^{2(n+j+1)}
)\right)\nonumber\\
& = \sum_{n\geq 1}\dfrac{(-1)^n 
q^{4k\binom{n}{2}+(2k(j+2)+1)n}(1-q^{2n})(1-q^{2(n+1)})\cdots 
(1-q^{2(n+j)})}{(1+q^{2n+1})(1+q^{2(n+1)+1})\cdots (1+q^{2(n+j)+1})}
q^{-2in}\nonumber\\
& \hphantom{=} + \sum_{n\geq 0}\dfrac{(-1)^n 
q^{4k\binom{n}{2}+(2k(j+2)+1)n}(1-q^{2(n+1)})\cdots 
(1-q^{2(n+j)})(1-q^{2(n+j+1)})}{(1+q^{2n+1})(1+q^{2(n+1)+1})\cdots 
(1+q^{2(n+j)+1})} q^{2(i-1)(n+j+1)}\nonumber\\
& = -\sum_{n\geq 0}\dfrac{(-1)^n 
q^{4k\binom{n}{2}+4kn+(2k(j+2)+1)(n+1)}(1-q^{2(n+1)})(1-q^{2(n+2)})\cdots 
(1-q^{2(n+j+1)})}{(1+q^{2(n+1)+1})(1+q^{2(n+2)+1})\cdots 
(1+q^{2(n+j+1)+1})} q^{-2i(n+1)}\nonumber\\
& \hphantom{=} + \sum_{n\geq 0}\dfrac{(-1)^n 
q^{4k\binom{n}{2}+(2k(j+2)+1)n}(1-q^{2(n+1)})\cdots 
(1-q^{2(n+j)})(1-q^{2(n+j+1)})}{(1+q^{2n+1})(1+q^{2(n+1)+1})\cdots 
(1+q^{2(n+j)+1})} q^{2(i-1)(n+j+1)}\nonumber\\
& = \sum_{n\geq 0}\dfrac{(-1)^n 
q^{4k\binom{n}{2}+(2k(j+2)+2i-1)n}(1-q^{2(n+1)})\cdots 
(1-q^{2(n+j+1)})}{(1+q^{2n+1})(1+q^{2(n+1)+1})\cdots 
(1+q^{2(n+j+1)+1})}\nonumber\\
& \hspace{3em} \cdot
\left(-q^{4kn+2k(j+2)+1-2in+2n-2i(n+1)}(1+q^{2n+1})+q^{-2in+2n+2(i-1)(n+j+1)}
(1+q^{2(n+j+1)+1})\right).
\end{align}
We can simplify the term in parentheses in the last line of 
(\ref{inductioncalc2}) to obtain
\begin{align*}  
-&q^{2n(2k-2i+1)+2k(j+1)+2k-2i+1}(1+q^{2n+1})+q^{2(j+1)(i-1)}(1+q^{2(n+j+1)+1}
)\nonumber\\
& = 
q^{2(j+1)(i-1)}(-q^{(2n+1)(2k-2i+1)+2(j+1)(k-i+1)}(1+q^{2n+1})+1+q^{2(n+j+1)+1}
)\nonumber\\
& = 
q^{2(j+1)(i-1)}(1-q^{2(k-i+1)(2n+1+j+1)}+q^{2(n+j+1)+1}(1-q^{
2(2n+1)(k-i)+2(j+1)(k-i)}))\nonumber\\
& = 
q^{2(j+1)(i-1)}(1-q^{2(k-i+1)(2n+j+2)}+q^{2(n+j+1)+1}(1-q^{2(k-i)(2n+j+2)})).
\end{align*}
The second sum in (\ref{twosums}) is exactly the same as the first except that 
$i$ is replaced by $i-1$, so we now see that (\ref{twosums}) becomes
\begin{align}\label{finalformula}
& G_{(k-1)(j+1)+i}(q)  +q^{-1} G_{(k-1)(j+1)+i-1}(q)\nonumber\\
& = \dfrac{1}{F(q)}\sum_{n\geq 0}\dfrac{(-1)^n 
q^{4k\binom{n}{2}+(2k(j+2)+2i-1)n}(1-q^{2(n+1)})\cdots 
(1-q^{2(n+j+1)})}{(1+q^{2n+1})(1+q^{2(n+1)+1})\cdots 
(1+q^{2(n+j+1)+1})}\nonumber\\
& \hspace{6em} 
\cdot\left(1-q^{2(k-i+1)(2n+j+2)}+q^{2(n+j+1)+1}(1-q^{2(k-i)(2n+j+2)}
)\right)\nonumber\\
& \hphantom{=} + 
\dfrac{q^{2j+1}q^{2(j+1)(i-2)}}{q^{2(j+1)(i-1)}F(q)}\sum_{n\geq 0}\dfrac{(-1)^n 
q^{4k\binom{n}{2}+(2k(j+2)+2(i-1)-1)n}(1-q^{2(n+1)})\cdots 
(1-q^{2(n+j+1)})}{(1+q^{2n+1})(1+q^{2(n+1)+1})\cdots 
(1+q^{2(n+j+1)+1})}\nonumber\\
& \hspace{11em}\,\, \cdot
\left(1-q^{2(k-i+2)(2n+j+2)}+q^{2(n+j+1)+1}(1-q^{2(k-i+1)(2n+j+2)})\right). 
\end{align}
The theorem now follows because the second term on the right-hand side of 
(\ref{finalformula}) is $q^{-1} G_{(k-1)(j+1)+i-1}(q)$ by the inductive 
hypothesis on $i$.
\end{proof}

\begin{rema}\label{rem:uniqueF}
It is important to note that the common factor $F(q)$ 
plays no role in the proof here, except for
the identification with the original $G_i(q)$, $i=1,\dots,k$. 
The factor $F(q)$ could have been replaced
with any nonzero formal power series in $q$, 
and every step of the proof would 
have been identical (beyond the identification with the original 
$G_i\left(q\right)$, $i=1,\dots,k$),
and equivalent to the existing step.
However, $F(q)$ is crucial for the ``Empirical Hypothesis,'' which, 
in fact, as we shall see,
uniquely determines this factor.
\end{rema}

\begin{rema}
The nested induction (on $i$ as well as on $j$) is a feature 
of the proof of Theorem \ref{thm:closedform} that is 
not present in the proof of Theorem 2.1 of \cite{LZ}. 
\end{rema}

\section{The Empirical Hypothesis}
\label{sec:EH}
As a consequence of Theorem \ref{thm:closedform}, we are now in a position
to formulate and prove the Empirical Hypothesis, which will
be the main ingredient in the motivated proof of the 
G\"ollnitz-Gordon-Andrews identities.

\begin{theo}[Empirical Hypothesis]\label{EH}
For any $j\geq 0$ and $i=1,\dots,k$,
\begin{equation*}
G_{(k-1)j+i}(q)=1+q^{2j+1} \gamma(q) 
\end{equation*}
for some
\begin{equation*}
\gamma(q) \in \mathbb{C}[[q]].
\end{equation*}
\end{theo}
\begin{rema}
Note that since 
$G_{(k-1)j+k}(q)=G_{(k-1)(j+1)+1}(q)$, for $i=k$ Theorem \ref{EH} implies that we 
can write $G_{(k-1)j+k}(q)=1+q^{2j+3}\gamma(q)$ where $\gamma(q)$ is some formal power 
series.
\end{rema}
\begin{proof}
Since 
$$
4k\binom{n}{2}+(2k(j+1)+2i-1)n\geq 2k(j+1)+1\geq 2j+3
$$
for $n\geq 1$, it 
suffices to examine the $n=0$ term in (\ref{altsum}), which is
\begin{align}\label{nequals0}
& \dfrac{(1-q^2)\cdots 
(1-q^{2j})(1-q^{2(k-i+1)(j+1)}+q^{2j+1}(1-q^{2(k-i)(j+1)}))}{(1+q)(1+q^3)\cdots 
(1+q^{2j+1})\prod\limits_{m\,\not\equiv\, 2\,(\mathrm{mod}\,4)} (1-q^m)}.
\end{align}
Note first that 
$$
1-q^{2(k-i+1)(j+1)}+q^{2j+1}(1-q^{2(k-i)(j+1)})=1+q^{2j+1} 
h(q)
$$
where $h(q)$ is a polynomial in $q$. Now, the remaining factors in the 
numerator of (\ref{nequals0}) can be expressed as:
\begin{align}
(1-q^2)\cdots (1-q^{2j})& =\left(\prod_{\substack{1\leq m\leq 2j\\m\,\equiv\, 
0\,(\mathrm{mod}\,4)}} (1-q^m)\right)\left(\prod_{\substack{1\leq m\leq 
j\\m\,\mathrm{odd}}} (1-q^m)(1+q^m)\right),
\end{align}
while the denominator of (\ref{nequals0}) can be rewritten
\begin{align*}
& (1+q)(1+q^3)\cdots (1+q^{2j+1})\prod_{m\,\not\equiv\, 2\,(\mathrm{mod}\,4)} 
(1-q^m)=\nonumber\\
& \left(\prod_{\substack{1\leq m\leq 2j+1\\ m\,\mathrm{odd}}} 
(1+q^m)(1-q^m)\right)\left(\prod_{\substack{1\leq m\leq 2j\\ m\,\equiv\, 
0\,(\mathrm{mod}\,4)}} (1-q^m)\right)\left(\prod_{\substack{m>2j+1\\ 
m\,\not\equiv\, 
2\,(\mathrm{mod}\,4)}} (1-q^m)\right)=\nonumber\\
& \left(\prod_{\substack{1\leq m\leq j\\ m\,\mathrm{odd}}} 
(1+q^m)(1-q^m)\right)\left(\prod_{\substack{j< m\leq 2j+1\\ m\,\mathrm{odd}}} 
(1-q^{2m})\right)\left(\prod_{\substack{1\leq m\leq 2j\\ m\,\equiv\, 
0\,(\mathrm{mod}\,4)}} (1-q^m)\right)\left(\prod_{\substack{m>2j+1\\ 
m\,\not\equiv\, 
2\,(\mathrm{mod}\,4)}} (1-q^m)\right).
\end{align*}
Thus by cancellation we see that (\ref{nequals0}) becomes
\begin{align*}
\dfrac{1+q^{2j+1} h(q)}{\left(\prod\limits_{\substack{j< m\leq 2j+1\\ 
m\,\mathrm{odd}}} (1-q^{2m})\right)\left(\prod\limits_{\substack{m>2j+1\\ 
m\,\not\equiv\, 
2\,(\mathrm{mod}\,4)}} (1-q^m)\right)}=1+q^{2j+1} g(q)
\end{align*}
where $g(q)$ is a formal power series. 
\end{proof}

\begin{rema}
It is clear from the proof of Theorem \ref{EH} that $F(q)$ 
(recall \eqref{eq:Fq}) is the unique formal 
power series in the denominator of $G_{(k-1)j+i}(q)$
that yields the Empirical Hypothesis.
\end{rema}

\begin{rema}\label{rem:SEH}
Analyzing the polynomial $h(q)$ in the proof above, we can strengthen the 
Empirical Hypothesis as follows: Observe that
\begin{equation*}
 h(q)=-q^{2(k-i)(j+1)+1}+1-q^{2(k-i)(j+1)},
\end{equation*}
so that for $i\neq k$ and $j\ge 0$,
\begin{align}
G_{(k-1)j+i}(q) = 1+q^{2j+1}+\cdots, \label{SEHi<k} 
\end{align}
and for $i=k$,
\begin{align}
G_{(k-1)j+k}(q) = G_{(k-1)(j+1)+1}(q)=1+q^{2j+3}+\cdots \label{SEHi=k}.
\end{align}
We call \eqref{SEHi<k} and \eqref{SEHi=k} collectively the \textit{Strong 
Empirical Hypothesis}.
\end{rema}

\begin{rema}\label{rem:WEH}
As we shall see in the proof of Theorem \ref{thm:GGA} below, the only form of 
the Empirical Hypothesis that is logically required for the proof of the \gga 
identities is a weaker one, which states that for any positive integer $l$, 
there exists a positive integer $f(l)$ such that 
\begin{align*}
G_l(q) \in 1+q^{f(l)}\mathbb{C}[[q]] 
\end{align*}
with
\begin{align*}
\lim_{l\rightarrow\infty}f(l)=\infty. 
\end{align*}
We call this the \textit{Weak Empirical Hypothesis}.
\end{rema}

\begin{rema}\label{rem:EHfromSums}
All of the successively sharper forms of Empirical Hypothesis follow easily from 
the combinatorial sum sides of the \gga identities, but the point of Theorems 
\ref{thm:closedform} and \ref{EH} is to obtain them from the product sides.
\end{rema}

\section{Matrix interpretation and consequences}
\label{sec:matrixinterp}
As in \cite{LZ}, we now study the recursions \eqref{recursion} using a matrix 
approach, which is suggested by rewriting \eqref{recursion} in the form
\begin{equation}\label{recursion2}
 q^{-1} 
G_{(k-1)j+i-1}(q)+G_{(k-1)j+i}(q)=q^{2j(i-1)}G_{(k-1)(j-1)+k-i+1}(q)-q^{2j(i-1)}
G_{(k-1)(j-1)+k-i+2}(q)
\end{equation}
for $j\geq 1$ and $i=2,\dots,k$. Recall also the edge-matching tautology
\begin{equation}\label{edgematching}
 G_{(k-1)j+1}(q)=G_{(k-1)(j-1)+k}(q)
\end{equation}
for $j\geq 1.$

Define the vector
 \begin{equation*}
{\bf G}_{(0)} = 
\left[ \begin{array}{c}
G_{1}(q) \\
\vdots \\
G_{k}(q)
\end{array} \right]
\end{equation*}
and more generally for each $j\geq 0$ define the vector
 \begin{equation*}
{\bf G}_{(j)} = 
\left[ \begin{array}{c}
G_{(k-1)j+1}(q) \\
\vdots \\
G_{(k-1)j+k}(q)
\end{array} \right].
\end{equation*}
For each $j \ge 1$ set
\begin{equation*}
{\bf B}_{(j)} = 
\left[ \begin{array}{cccccc}
0           & 0            & \cdots   & 0       & 0        & 1       \\
0           & 0            & \cdots   & 0       & q^{-2j}   & -q^{-2j} \\
0           & 0            & \cdots   & q^{-4j} & -q^{-4j} & 0       \\
\vdots      & \vdots       & \swarrow & \vdots  & \vdots   & \vdots  \\
0           & q^{-2(k-2)j}  & \cdots   & 0       & 0        & 0       \\
q^{-2(k-1)j} & -q^{-2(k-1)j} & \cdots   & 0       & 0        & 0
\end{array} \right]
\end{equation*}
and set
\begin{equation*}
{\bf C} = 
\left[ \begin{array}{cccccc}
1           & 0            & 0 & \cdots          & 0        & 0       \\
q^{-1}           & 1       & 0    & \cdots          & 0  & 0 \\
0           & q^{-1}       & 1    & \cdots    & 0 & 0       \\
\vdots      & \vdots     & \vdots & \searrow   & \vdots   & \vdots  \\
0           &  0 & 0 & \cdots          & 1        & 0       \\
0 & 0 & 0 & \cdots          & q^{-1}        & 1
\end{array} \right] .
\end{equation*}
Then (\ref{recursion2}) and (\ref{edgematching}) can be combined in the single 
matrix equation
\begin{equation}\label{matrixrecursion}
 \mathbf{C}\mathbf{G}_{(j)}=\mathbf{B}_{(j)} \mathbf{G}_{(j-1)}
\end{equation}
for each $j\geq 1$.

It is easy to check that the inverse of 
$\mathbf{B}_{(j)}$ is the matrix
\begin{equation}\label{Aj}
{\bf A}_{(j)} = {\bf B}_{(j)}^{-1} = 
\left[ \begin{array}{ccccc}
1 & q^{2j} & q^{4j} & \cdots  & q^{2( k-1) j} \\
\vdots & \vdots &  \vdots & \swarrow  & \vdots \\
1 & q^{2j} & q^{4j}  & \cdots & 0 \\
1 & q^{2j} & 0 &  \cdots & 0 \\
1 & 0 & 0 & \cdots  & 0
\end{array} \right].
\end{equation}
Setting 
$$
\mathbf{A}_{(j)}'=\mathbf{A}_{(j)}\mathbf{C}
$$
for each $j\geq 1$, we have:
\begin{equation}\label{A'j}
 \mathbf{A}'_{(j)}=\mathbf{A}_{(j)}\mathbf{C} = \left[ \begin{array}{ccccc}
1+q^{2j-1} & q^{2j}+q^{4j-1}  & \cdots  & q^{2(k-2)j}+q^{2(k-1)j-1} & q^{2( 
k-1) 
j} \\
1+q^{2j-1} & q^{2j}+q^{4j-1}   & \cdots & q^{2(k-2)j}& 0 \\
\vdots & \vdots  & \swarrow & \vdots  & \vdots \\
1+q^{2j-1} & q^{2j}  &  \cdots & 0 &0 \\
1 & 0 & \cdots & 0 & 0
\end{array} \right].
\end{equation}

For the remainder of this paper, we fix an integer
\begin{equation} \label{eq:J}
J \geq 0,
\end{equation}
which will indicate a ``starting shelf'' that need not be 
the 0$^\text{th}$ shelf. We will use 
\eqref{matrixrecursion} to 
express ${\bf G}_{(J)}$ in terms of the ${\bf G}_{(j)}$, $j \ge J$,
and we will then use the Empirical Hypothesis to 
determine an expression for ${\bf G}_{(J)}$ 
which, in the case $J=0$, will be different from the 
original definition \eqref{Gidef}.

For any $j\geq J$, repeated application of 
\eqref{matrixrecursion} gives us
\begin{equation*}
{\bf G}_{(J)} = 
\mathbf{A}_{(J+1)}'\cdots\mathbf{A}_{(j)}'\mathbf{G}_{(j)}.
\end{equation*}
That is,
\begin{equation}\label{GJintermsofGj}
\left[ \begin{array}{c}
G_{(k-1)J+1}(q) \\
\vdots \\
G_{(k-1)J+k}(q)\end{array} \right] = 
\mathbf{A}_{(J+1)}'\cdots\mathbf{A}_{(j)}'\mathbf{G}_{(j)}={}^J\mathbf{h}^{(j)}
\mathbf
{G}_{(j)},
\end{equation}
where 
\begin{equation*}
{}^J\mathbf{h}^{(j)}=\mathbf{A}_{(J+1)}'\cdots\mathbf{A}_{(j)}'
\end{equation*}
for $j> J$. (We take ${}^J\mathbf{h}^{(J)}$ to be the empty product of 
matrices, i.e., the identity matrix.)

For each $j\geq J$ and $i=1,\dots,k$, define the row vector 
${}^J_i\mathbf{h}^{(j)}(q)$ to be the $i^\text{th}$ row of 
${}^J\mathbf{h}^{(j)}$, and for $l=1,\dots,k$, 
define ${}^J_i h^{(j)}_l(q)$ to be the $l^\text{th}$ component of  
${}^J_i\mathbf{h}^{(j)}(q)$. Then it is clear from (\ref{GJintermsofGj}) that 
for 
any $j\geq J$ and $i=1,\dots,k$,
\begin{equation}\label{G(k-1)J+i}
G_{(k-1)J+i}(q)  =  
{}^J_ih^{(j)}_1(q) G_{(k-1) j +1}(q) + \cdots + {}^J_ih^{(j)}_k(q) G_{(k-1) j + 
k}(q).
\end{equation}
Moreover, the ${}^J_i h^{(j)}_l(q)$ are polynomials in $q$ with nonnegative 
integral coefficients, since the entries of $\mathbf{A}'_{(j)}$ are polynomials
in $q$ with nonnegative integral coefficients for each $j$ by \eqref{A'j}. In 
addition, we have 
the following proposition, which will be needed for the combinatorial 
interpretation of the polynomials ${}^J_i h^{(j)}_l(q)$ in the next section:
\begin{propo}\label{hdef}
The polynomials ${}^J_i h^{(j)}_l(q)$ are determined by the initial  conditions
\begin{equation}\label{init}
{}^J_i h^{(J)}_l(q)=\delta_{i,l}
\end{equation}
and the recursions
\begin{equation}\label{hrecur}
{}^J_i h^{(j)}_l(q)=q^{2j(l-1)}\sum_{m=1}^{k-l+1} {}^J_i h^{(j-1)}_m (q) 
+q^{2jl-1}\sum_{m=1}^{k-l} {}^J_i h^{(j-1)}_m(q)
\end{equation}
for $j> J$.
\end{propo}
\begin{proof}
The initial conditions (\ref{init}) amount to the definition of 
${}^J\mathbf{h}^{(J)}$, 
and the recursions (\ref{hrecur}) are simply the equation
\begin{equation*}
 {}^J\mathbf{h}^{(j)}(q)={}^J\mathbf{h}^{(j-1)}(q)\mathbf{A}_{(j)}'
\end{equation*}
in component form.
\end{proof}

Let us note a few consequences of \eqref{G(k-1)J+i}, \eqref{hrecur} and the 
Empirical 
Hypothesis. We first explain what we mean by the limit of a sequence $\lbrace 
H_j(q)\rbrace_{j=1}^\infty$ of formal power series: We say that
\begin{equation*}
\lim_{j\rightarrow\infty} H_j(q)
\end{equation*}
\emph{exists} if for each $i\geq 0$, there is some $J_i>0$ such that the coefficients 
of $q^i$ in each series $H_j(q)$ for $j\geq J_i$ are equal. In other words, 
the limit exists if the coefficient of each power of $q$ in $H_j(q)$ stabilizes 
as $j\rightarrow\infty$. In this case, the limit of the $H_j(q)$ is the formal 
power series for which the coefficient of $q^i$ is the coefficient of $q^i$ in 
any $H_j(q)$ for which $j\geq J_i$. This stabilization condition is
appropriate since the coefficients of the power series we consider
are integral.

Now we see that if $l\neq 1$, because of the factors multiplying each of the 
summations in \eqref{hrecur},
\begin{equation}
\lim_{j\rightarrow \infty}{}^J_i{h}^{(j)}_l(q) = 0.
\label{eq:l>1-hJijl_is_0}
\end{equation}
In case $l=1$, let $t$ be a nonnegative integer. Then, because of \eqref{G(k-1)J+i}, 
\eqref{eq:l>1-hJijl_is_0}, and the Empirical Hypothesis,
even in its weakest form,
it follows that the coefficient of $q^t$ on the left-hand side of \eqref{G(k-1)J+i} is 
the same as the coefficient of $q^t$ in ${}^J_i{h}^{(j')}_1(q)$ 
for any $j'$ sufficiently large. This shows that 
the coefficient of any power of $q$ in ${}^J_i{h}^{(j)}_1(q)$ 
stabilizes as $j\rightarrow \infty$. Hence, 
\begin{equation*}
\lim_{j\rightarrow \infty}{}^J_i{h}^{(j)}_1(q)
\end{equation*}
exists, which we denote by
\begin{equation*}
{}^J_i{h}^{(\infty)}_1(q).
\label{hinfty}
\end{equation*}
We have in fact proved more, which we record as a theorem:

\begin{theo}\label{thm:GisHinfty}
For any $J \ge 0$ and $i=1,\dots,k$, 
\begin{align}
G_{(k-1)J+i}(q) &={}^J_ih^{(\infty)}_1(q).\label{eq:GisHinfty}
\end{align}
\end{theo}

\begin{proof}
It follows  from \eqref{eq:l>1-hJijl_is_0} and the Empirical Hypothesis that
\begin{align*}
G_{(k-1)J+i}(q) &= 
\left(\lim\limits_{j\rightarrow\infty}{}^J_ih^{(j)}_1\right)
\left(\lim\limits_{j\rightarrow\infty} G_{(k-1)j+1}\right)
+ \cdots +
\left(\lim\limits_{j\rightarrow\infty}{}^J_ih^{(j)}_k\right)
\left(\lim\limits_{j\rightarrow\infty} G_{(k-1)j+k}\right)
\nonumber\\
& = {}^J_ih^{(\infty)}_1\cdot 1 + 0 +\cdots + 0\\
&={}^J_ih^{(\infty)}_1.
\end{align*}
\end{proof}

\begin{rema}\label{rem:essenceofG=h^infty}
Note that ``sum sides'' of the \gga identities 
have not emerged in Theorem \ref{thm:GisHinfty} yet. 
In other words, we have yet to attach combinatorial meaning to the 
right-hand side of \eqref{eq:GisHinfty} 
(we will do this in the following section.).
Nonetheless, Theorem \ref{thm:GisHinfty} is the essence of the
``motivated proof'' since 
the left-hand side of \eqref{eq:GisHinfty} is concerned with
certain congruence conditions (recall Remark \ref{prodsideinterp}) and these 
congruence conditions are completely invisible on the 
right-hand side of \eqref{eq:GisHinfty}. 
Of course, this remark also holds in the context
of the Rogers-Ramanujan identities (cf. \cite{AB}) and in the more general
setting of the Gordon identities as well (cf. \cite{LZ} and Remark 
\ref{ehrenpreis}).

\end{rema}

\section{Combinatorial interpretation of the sequence of expressions for 
$G_l(q)$}
\label{sec:combinterp}

We are now ready to complete our motivated proof.
We will deduce combinatorial 
interpretations of the formal 
power series on all of the shelves, and in particular, 
we will prove the \gga identities, which correspond to 
the 0$^\text{th}$ shelf, i.e., $J=0$.
We proceed as follows: We first present a combinatorial interpretation of the 
polynomials ${}^J_i h^{(j)}_l(q)$,
then we take the limit $j\rightarrow\infty$, and 
finally we use Theorem \ref{thm:GisHinfty}.

\begin{propo}\label{hinterp}
For $j\geq J+1$, the polynomial ${}^J_i h^{(j)}_l(q)$ is the generating 
function
for partitions with parts $(b_1,\ldots,b_s)$ (with $b_p\geq b_{p+1}$), 
satisfying 
the following conditions: 
 \begin{enumerate}
  \item No odd parts are repeated,
  \item $b_p-b_{p+k-1}\geq 2$ if $b_p$ is odd,
  \item $b_p-b_{p+k-1} >2$ if $b_p$ is even,
  \item the smallest part $b_s > 2J$,
  \item there are no more than $k-i$ parts equal to $2J+1$ or $2J+2$, and
  \item $b_1\leq 2j$ and there are exactly $l-1$ parts equal to $2j$.
 \end{enumerate}
\end{propo}
\begin{proof}
 Use ${}^J_i \widetilde{h}_l^{(j)}(q)$ to denote the generating function for 
partitions satisfying the conditions of the proposition. It is enough to verify 
that ${}^J_i \widetilde{h}_l^{(J+1)}(q)={}^J_i h^{(J+1)}_l(q)$ and that for 
$j\geq J+2$, 
${}^J_i \widetilde{h}_l^{(j)}(q)$ satisfies the recursion (\ref{hrecur}). 
 
 In the case $j=J+1$, there are two partitions satisfying the conditions of 
the 
proposition if $l\leq k-i$: one with $l-1$ parts equal to $2J+2$ and no parts 
equal to $2J+1$, and one with $l-1$ parts equal to $2J+2$ and one part equal to 
$2J+1$. If $l=k-i+1$, only the partition with $l-1$ parts equal to $2J+2$ and 
no 
parts equal 
to $2J+1$ satisfies the conditions of the proposition, and if $l>k-i+1$, no 
partitions satisfy the conditions. Thus
 \begin{equation*}
{}_i \widetilde{h}_l^{(J+1)}(q)=\left\lbrace\begin{array}{ccc} 
q^{2(l-1)(J+1)}+q^{2l(J+1)-1} & \mathrm{if} & l\leq k-i+1\\ 
q^{2(J+1)(l-1)} & \mathrm{if} & l=k-i+1\\
0 & \mathrm{if} & l>k-i+1,\end{array}\right.
\end{equation*}
which agrees with ${}^J_i h^{(J+1)}_l(q)$ by inspection of \eqref{A'j}
(setting $j=J+1$).

Now consider $j\geq J+2$. For convenience, we say that a partition is of type 
$(k-1,2J,k-i)$ if it satisfies the first five conditions of the proposition. 
Then 
the partitions of type $(k-1,2J,k-i)$ having largest part at most $2j$ and 
having exactly $l-1$ parts equal to $2j$ can be divided into two sets: the ones 
having 
no parts equal to $2j-1$ and the ones having one part equal to $2j-1$. In the 
first case, condition 3 of the proposition implies that at most $k-l$ parts 
can be equal to $2j-2$, which means that the number of partitions in the first 
case are enumerated by the generating function $q^{2j(l-1)}\sum_{m=1}^{k-l+1} 
{}^J_i\widetilde{h}^{(j-1)}_m$. 

In the second case, the partitions in question have the form $((2j)^{l-1}, 2j-1, 
b_{l+1},\ldots b_s)$, where $(b_{l+1},\ldots b_s)$ is a partition of type 
$(k-1,2J,k-i)$ having largest part at most $2j-2$. By condition 3 of the 
proposition, at most $k-l-1$ 
parts of $(b_{l+1},\ldots, b_s)$ can equal $2j-2$. Conversely, we note that 
$(b_{l+1},\ldots, 
b_s)$ can be any partition of type $(k-i,2J,k-1)$ having largest part at most 
$2j-2$ and at most $k-l-1$ parts equal to $2j-2$: the fact that at most $k-l-1$ 
parts are equal to $2j-2$ for 
$(b_{l+1},\ldots, 
b_s)$ implies that $b_k<2j-2$, which in turn implies conditions 2 and 3 for the 
larger
partition $((2j)^{l-1}, 2j-1, b_{l+1},\ldots, b_s)$, while conditions 1, 4, 5 
and 6 for the larger partition 
follow immediately from the same conditions
 for $(b_{l+1},\ldots,b_s)$. Thus 
the partitions in the second case are 
enumerated 
by the generating function $q^{2j(l-1)+2j-1}\sum_{m=1}^{k-l} 
{}^J_i\widetilde{h}^{(j-1)}_m$.

In conclusion, we see that for any $i,l=1,\dots,k$,
\begin{equation*}
 {}^J_i\widetilde{h}^{(j)}_l(q)=q^{2j(l-1)}\sum_{i=1}^{k-l+1} 
{}^J_i\widetilde{h}^{(j-1)}_m(q) +q^{2j(l-1)+2j-1}\sum_{m=1}^{k-l} 
{}^J_i\widetilde{h}^{(j-1)}_m(q),
\end{equation*}
which is the recursion (\ref{hrecur}). This proves the proposition.
\end{proof}

In the case $l=1$, condition 6 in Proposition \ref{hinterp} just says that the 
largest part of the partition is strictly less than $2j$. Thus ${}^J_i 
h^{(j)}_1(q)$ is the generating function for partitions satisfying conditions 
$1-5$ of Proposition \ref{hinterp} with largest part less than or equal to 
$2j-1$. Since each coefficient of ${}^J_i h^{(\infty)}_1(q)$ agrees with the 
corresponding coefficient of ${}^J_i h^{(j)}_1(q)$ for sufficiently large $j$, 
we have:
\begin{propo}\label{HinfinityCombinatorics}
The formal power series ${}^J_i h^{(\infty)}_1(q)$ is the generating function for 
partitions satisfying conditions $1-5$ of Proposition \ref{hinterp}, that is, 
${}^J_i h^{(\infty)}_1(q)$ is the generating function for partitions of type 
$(k-1,2J,k-i)$.
\end{propo}

\begin{rema}\label{EHfromSumSide}
 Note that it is obvious from conditions 1 and 4 of Proposition \ref{hinterp} 
that the Strong Empirical Hypothesis of Remark \ref{rem:SEH} holds for the 
series ${}^J_i h^{(\infty)}_1(q)$ (recall Remark \ref{rem:EHfromSums}).
\end{rema}

Now applying Theorem \ref{thm:GisHinfty}, we see that $G_{(k-1)J+i}(q)$ for any 
$J\geq 0$ and $i=1,\dots,k$ is the generating function for partitions of type 
$(k-1,2J,k-i)$. In the special case $J=0$, we see that Remark \ref{prodsideinterp}, Theorem 
\ref{thm:GisHinfty} and Proposition \ref{HinfinityCombinatorics} immediately 
imply the \gga identities:

\begin{theo}\label{thm:GGA}
For any $k\geq 1$ and $i=1,\dots,k$, the number of partitions of a nonnegative integer $n$ with parts not 
congruent to $2$ mod $4$ and not congruent to $2k\pm(2i-1)$ mod $4k$ is equal 
to the number of partitions $(b_1,\ldots, b_s)$ of $n$ such that no odd parts 
are 
repeated, $b_p-b_{p+k-1}\geq 2$ if $b_p$ is odd, $b_p-b_{p+k-1}>2$ if $b_p$ is 
even, and there are no more than $k-i$ parts equal to $1$ or $2$. 
\end{theo}

\begin{rema}\label{rem:altmech}
The main ingredients in the proof of Theorem
\ref{thm:GGA} are the recursive definition of the $G_l(q)$'s 
and the Weak Empirical Hypothesis (cf. Remark \ref{rem:WEH}),
but the determination
of closed-form expressions for the $G_l(q)$'s could in principle be
replaced with any other mechanism that
yields the Weak Empirical Hypothesis.
\end{rema}

\begin{rema}\label{rem:robinson} (Cf.\ Remark 2.1 in \cite{LZ}.)
As discussed in \cite{AB}, \cite{R} and \cite{A4}, we can give
an alternate, shorter proof of Theorem \ref{thm:GGA} using
only the Empirical Hypothesis and without using Proposition \ref{hinterp}.
Let $J_1(q),J_2(q),\dots$ be a sequence of formal power series in $q$ 
with constant term $1$ satisfying the recursions \eqref{recursion}
for $j\ge 1$ (with $J_p(q)$'s in place of the $G_p(q)$'s), and suppose that
the Empirical Hypothesis holds for $J_1(q),J_2(q),\dots$. 
In fact, $J_1(q),J_2(q),\dots,J_k(q)$, and therefore all $J_1(q),J_2(q),\dots$,
are uniquely determined by the recursions \eqref{recursion} and the Empirical 
Hypothesis. 
If one recursively builds such a sequence of formal power series, say 
$S_1(q),S_2(q),\dots$,
using \eqref{recursion} and starting from the combinatorial generating 
functions for the sum sides of the \gga identities, then the Empirical Hypothesis 
is easily seen to hold for $S_1(q),S_2(q),\dots$.
Since by Theorem \ref{EH} the Empirical Hypothesis holds
for the sequence $G_1(q),G_2(q),\dots$, we must have that $S_l(q) = 
G_l(q)$ for all $l\ge 1$, by the uniqueness. For $l=1,\dots,k$ this precisely gives the \gga 
identities.
\end{rema}

%%%%%%%%%%%%%%%%%%%%%%%%%%%%%%%%%%%%%%%%%%%%%%

\appendix

\section*{Appendices}

\section{An $(x,q)$-dictionary for the \gga identities}
\label{app:xqforGGA}
The aim of this appendix is to compare the motivated proof of the \gga 
identities given in this paper and the proof of these identities given in 
Chapter 7 of \cite{A3}. Our main observation is the following:
After setting up the correct dictionary between suitable specializations of the 
$J_{k,i}(a, x, q)$ functions defined in (7.2.2) of \cite{A3} and the $G_l(q)$'s 
defined in Section \ref{sec:shelf0}, the steps in the proof of Theorem 
\ref{thm:closedform} can be matched with those in the proofs of Lemmas 7.1 and 
7.2 in \cite{A3}.

Let 
\begin{align}
(y)_n&=(1-y)\cdots(1-yq^{n-1})\label{notation:(y)_n}, \\
(y)_\infty&=\prod_{m\ge0}(1-yq^m)\label{notation:(y)_infty}.
\end{align}
We will use this notation in this appendix and in Appendix \ref{app:xqforGRR}.

Recall the definition (7.2.2) from \cite{A3}:
\begin{align}
J_{k,i}&(a,x,q)=\nonumber\\
&  \sum_{n\ge0}
\dfrac{x^{kn}q^{kn^2+kn+n-in}a^n(1-x^iq^{(2n+1)i})(axq^{n+2})_\infty(a^{-1}
)_\infty}
{(q)_n(xq^{n+1})_\infty}
\nonumber\\
& \hphantom{=} -
xqa\sum_{n\ge0}
\dfrac{x^{kn}q^{kn^2+kn+n-(i-1)n}a^n(1-x^{i-1}q^{(2n+1)(i-1)})(axq^{n+2}
)_\infty(a^{-1})_\infty}
{(q)_n(xq^{n+1})_\infty} 
\label{eq:7.2.2ofA2},
\end{align}
for any $k\ge 2$ and $i=0,\dots,k+1$.
Note that for the ``edge cases'' $i=0$ and $i=k+1$,
$J_{k,i}(a, x, q)$ is still a formal power series in 
the variables $x$ and $q$, i.e., no negative powers of 
$x$ or $q$ appear.

Now, specializing $q\mapsto q^2$, $a\mapsto -q^{-1}$ and 
$x\mapsto q^{2j}$, we get:
\begin{align}
J&_{k,i}(-q^{-1}, q^{2j}, q^2)=\nonumber\\
&  \sum_{n\ge0}
\dfrac{q^{2(jkn + 
kn^2+kn+n-in)}(-1)^nq^{-n}(1-q^{2(j+2n+1)i})(-q^{2(j+n+2)-1})_\infty(-q)_\infty}
{(q^2)_n(q^{2(j+n+1)})_\infty}
\nonumber\\
& \hphantom{=} +
q^{2(j+1)-1}\sum_{n\ge0}
\dfrac{q^{2(jkn + 
kn^2+kn+n-(i-1)n)}(-1)^nq^{-n}(1-q^{2(j+2n+1)(i-1)})(-q^{2(j+n+2)-1}
)_\infty(-q)_\infty}
{(q^2)_n(q^{2(j+n+1)})_\infty}
\label{eq:7.2.2ofA2spec}.
\end{align}
Combining the sums and multiplying and dividing appropriately so that the
denominator is
$$(1+q^{2n+1})(1+q^{2(n+1)+1})\cdots(1+q^{2(n+j)+1})
F(q)$$
(recall the definition \eqref{eq:Fq} of $F(q)$), 
we arrive at:
\begin{align*}
J_{k,i}(-q^{-1},q^{2j},q^2)=
\dfrac{1}{F(q)}
\sum_{n\geq 0}& \dfrac{(-1)^n q^{4k\binom{n}{2}+(2k(j+2)-2i+1)n}(1-q^{2(n+1)})
\cdots(1-q^{2(n+j)})}
{(1+q^{2n+1})(1+q^{2(n+1)+1})\cdots(1+q^{2(n+j)+1})}\nonumber\\
&\cdot \left(1-q^{2i(2n+1+j)}+q^{2(n+j)+1}(1-q^{2(i-1)(2n+1+j)})\right).
\end{align*}
Comparing with \eqref{altsum},
we obtain the following dictionary:
\begin{equation}\label{eq:GGAdict}
G_{(k-1)j+i}(q)=J_{k,k-i+1}(-q^{-1},q^{2j},q^2).
\end{equation}

The recursive definition of $G_l(q)$ in \eqref{recursion}
is motivated by Lemma 7.2 of \cite{A3}, which states that
\begin{equation*}
J_{k,i}(a,x,q)-J_{k,i-1}(a,x,q)=(xq)^{i-1}(J_{k,k-i+1}(a,xq,q)-aJ_{k,k-i+2}(a,xq
,q)).
\end{equation*}
Making the replacements $q\mapsto q^2$, $a\mapsto -q^{-1}$ 
and $x\mapsto q^{2j}$ for $j\geq 0$, we obtain
\begin{align*}
& J_{k,i}(-q^{-1},q^{2j},q^2)-J_{k,i-1}
(-q^{-1}, q^{2j}, q^2)=\nonumber\\
& q^{2(j+1)(i-1)}(J_{k,k-i+1}(-q^{-1}, q^{2(j+1)},q^2)
+q^{-1} J_{k,k-i+2}(-q^{-1},q^{2(j+1)},q^2)).
\end{align*}
Rearranging the terms, we see that
\begin{align}\label{eq:lemma7.2specialized}
J_{k,k-i+1}(-q^{-1},q^{2(j+1)},q^2) = 
&\,\dfrac{J_{k,i}(-q^{-1},q^{2j},q^2)-
J_{k,i-1}(-q^{-1},q^{2j},q^2)}{q^{2(j+1)(i-1)}}\nonumber\\
&-q^{-1} J_{k,k-i+2}(-q^{-1},q^{2(j+1)},q^2).
\end{align}
Under \eqref{eq:GGAdict}, edge-matching for the $G_l(q)$'s amounts to 
showing that for $j\geq 0$,
\begin{equation*}
J_{k,1}(-q^{-1},q^{2j},q^2)=J_{k,k}(-q^{-1},q^{2(j+1)},q^2).
\end{equation*}

\begin{rema}\label{rem:motivated_to_xq}
It is not hard to reverse-engineer the procedure of this section. 
That is, one can replace judiciously chosen instances of pure powers 
of $q$ with appropriate powers of $x$ in the proof of Theorem \ref{thm:closedform},
and in this way our motivated proof would yield an ``$(x,q)$-proof''
similar in spirit to the ones in Chapter 7 of \cite{A3}.
A similar observation was already made in Section 5 of \cite{AB}.
\end{rema}

%%%%%%%%%%%%%%%%%%%%%%%%%%%%%%%%%%%%%%%%%%%%%%

\section{Some remarks about \cite{LZ}}
\label{app:LZremarks}

The developments in the present paper suggest a number of enhancements
to and comments on various aspects of \cite{LZ},
in which a motivated proof of Gordon's generalization of
the Rogers-Ramanujan identities was given.  We proceed to
describe these, referring in each case to particular sections,
formulas, and/or results in \cite{LZ}.

\subsection{The start of the induction in the proof of Theorem 2.1
of \cite{LZ}}

The first half of Theorem 2.1 of \cite{LZ} entails proving 
equation (2.7) of \cite{LZ} (recalled below in \eqref{eq:LZ2.7}), which gives
closed forms for recursively defined series $G_{(k-1)j+i}(q)$, where
$j\ge 0$ and $i=1,\dots,k$. (Note that the series $G_{(k-1)j+i}(q)$ here are not 
the same as the $G_{(k-1)j+i}(q)$'s defined in Section \ref{sec:shelf0}.)
In this subsection, we comment on
the step corresponding to $j=0$ of the inductive proof of Theorem 2.1. 
For the inductive step corresponding to $j=0$, 
the proof rests on equation (2.5) of \cite{LZ},
and for all higher $j$'s, the proof gives a variant of the argument in (2.5).
However, the steps in the proof in fact hold for $j=0$ also, thereby giving
a slightly different proof of (2.5).
In other words, an alternate way of presenting Theorem 2.1 would have been
to omit the calculations in and around (2.5) and take $j\ge 0$ in the 
inductive
step of the proof of Theorem 2.1. We now elaborate on this.

We first recall a few formulas from Section 2 of \cite{LZ}.
For $i=1,\dots,k$, formula (2.2) in \cite{LZ} is:
$$G_i(q)
=
\frac{ 1+ \sum_{\lambda \geq 1}
(-1)^\lambda q^{ (2k+1) {\binom{\lambda}{2}} + (k-i+1) \lambda} (1 + q^{(2i-1) 
\lambda })}{ \prod_{n \geq 1} (1-q^n)},$$
and formula (2.3) is:
$$G_i(q) =
\frac{ \sum_{\lambda \geq 0}
(-1)^\lambda q^{ (2k+1) {\binom{\lambda}{2}} + (k+i) \lambda} (1 - q^{ (k-i+1) 
(2 
\lambda+1)} )}
{ \prod_{n \geq 1} (1-q^n)}.$$

Equation (2.2) of \cite{LZ} was obtained by specializing formal variables in 
the 
Jacobi triple product identity and was used to obtain the
closed-form expression for the series $G_{k-1+i}(q)$ for $i=2,\dots,k$
as in (2.5) of \cite{LZ}:
$$ G_{k -1 +i}(q)
=  \frac{
\sum_{\lambda \geq 0}
(-1)^{ \lambda }  q^{ (2k+1) {\binom{\lambda}{2}} + (2k + i ) \lambda}
(1 - q^{\lambda+1}) ( 1 - q^{ (k - i + 1 ) (2\lambda + 2) } )
}{ \prod_{n \geq 1} (1-q^n) }
,$$
where the series $G_{k-1+i}(q)$ are defined recursively by:
$$
G_{k -1 +i}(q) = \frac{G_{k -i +1 } - G_{k- i + 2}}{q^{i-1} }
$$
for $i = 2, \dots, k$.

We remark that although it was natural to use 
(2.2) to prove formula (2.5), it could have been omitted and 
all the computations could have been done starting with (2.3) instead of (2.2). 
In fact, formula (2.5) is a special case corresponding to  $j=1$ of formula 
(2.7) in Theorem 2.1 in \cite{LZ}:
\begin{equation*}
G_{(k-1) j + i}(q) = \frac{
\sum_{\lambda \geq 0} (-1)^\lambda
q^{ (2k +1) {\binom{\lambda}{2}} + [ k (j+1) + i ] \lambda }
(1- q^{\lambda +1} ) \cdots (1 - q^{\lambda + j} )
(1 - q^{ ( k - i + 1) (2 \lambda + j + 1) } )
}{\prod_{n \geq 1} (1-q^n) }. \label{eq:LZ2.7}
\end{equation*}
The proof of (2.7) used the analogue of (2.3) 
instead of the analogue of (2.2). 
As we mentioned, the proof of Theorem 2.1 refers to 
(2.5) for the truth of the 
case corresponding to $j=0$, but the steps in the proof  
hold for $j=0$ as well. 
As promised, the case $j=0$ gives an alternate proof of (2.5),
this time using (2.3) instead of (2.2).

\subsection{Concerning the Empirical Hypothesis}
Let us recall the formulation and proof (using Theorem 2.1 of \cite{LZ})
of the Empirical Hypothesis from \cite{LZ}. 
For $j\geq 0$ and $i=1,\dots,k$,
\begin{eqnarray} \label{G_iclosedform}
G_{(k-1) j + i}(q) \! \!
& = & \!\! 
\frac{ 1 - q^{ ( k-i + 1) ( j+1) } }
{ (1 - q^{ j+1} ) ( 1- q^{ j+2} ) \cdots}
 \\
&  & \!\! 
 +\, \frac{
\sum_{\lambda \geq 1} (-1)^\lambda
q^{ (2k +1) {\binom{\lambda}{2}} + [ k (j+1) + i ] \lambda }
(1- q^{\lambda +1} ) \cdots (1 - q^{\lambda + j} )
(1 - q^{ ( k - i + 1) (2 \lambda + j + 1) } )
}{ \prod_{n \geq 1} (1-q^n) }.
 \nonumber \\
& = & \!\! 
1 + q^{j+1} \gamma_i^{(j+1)} (q)  \quad \text{ if }\: 1  \leq i \leq k-1
\nonumber \\
& \text{or } & \!\! 
1 + q^{j+2} \gamma_k^{(j+2)} (q) \quad \text{ if }\: i = k,
\nonumber
\end{eqnarray}
where
\[
\gamma^{(j)}_i (q) \in \mathbb{C}[[q]].
\]

In fact, \eqref{G_iclosedform} implies more:
For $i=1,\dots,k-1$,
$$\gamma_i^{(j+1)}(q)\in 1+q\mathbb{C}[[q]],$$
so that for $i=1,\dots,k-1$,
$$G_{(k-1)j+i}(q)=1+q^{j+1}+\cdots$$
and for $i=k$,
$$\gamma_i^{(j+2)}(q)\in 1+q\mathbb{C}[[q]]$$
so that
$$G_{(k-1)j+k}(q)=1+q^{j+2}+\cdots.$$

Of course, analogues of Remarks \ref{rem:WEH}, \ref{rem:EHfromSums}, and 
\ref{rem:altmech}
in the present work are true in the setting of the  Gordon identities:

\begin{rema}\label{rem:weakEHisenough}
(Cf.\ Remark \ref{rem:WEH}.)
Of the successively sharper forms of the
Empirical Hypothesis we have arrived at, the only form 
that is really needed for the motivated proof is the 
weakest one, which states that for all positive integers $l$
there exists a positive integer $f(l)$ such that
$$G_l(q) \in 1+q^{f(l)}\mathbb{C}[[q]]$$
with
$$ \lim_{l\rightarrow \infty}f(l) = \infty.$$
We refer to this weakest form as the Weak Empirical Hypothesis.
\end{rema}

\begin{rema}\label{rem:altroute}
(Cf.\ Remark \ref{rem:altmech}.)
The main facts used in the proof of the Gordon
identities are the recursive definition of the $G_l(q)$ and the Empirical 
Hypothesis. Therefore, in principle, the
calculation of the closed-form expressions for the series $G_l(q)$ could have 
been replaced
with any other mechanism that entails the Empirical Hypothesis.
\end{rema}

\begin{rema}\label{rem:EHobvfromsum}
(Cf.\ Remark \ref{rem:EHfromSums}.)
The (strongest form of the) Empirical Hypothesis is obvious from the
sum sides of the Gordon identities. The point
is to obtain it from only the corresponding product sides.
\end{rema}

\subsection{$J$-generalization of statements in Proposition 2.1 of \cite{LZ}}

In the rest of this appendix, we give a series of 
$J$-generalizations (as in the body of the present paper;
recall \eqref{eq:J})
of the propositions and theorems in [LZ].
It is worth recalling that in \cite{AB} --- 
the special case of the Rogers-Ramanujan identities, i.e.,
$k=2$ --- the $J=3$ case gave an answer to Ehrenpreis's question
(without resorting to the Rogers-Ramanujan identities themselves, 
demonstrating that the subtraction of the products, $G_1(q)-G_2(q)$,
has nonnegative coefficients), 
and the $J=1$ and $J=2$ cases led to a proof of the (two)
Rogers-Ramanujan identities themselves.

Proposition 2.1 in [LZ] gives a recursion relation among the 
coefficients ${}_ih_1^{(j)}(q)$ 
that arise when we express each $G_i(q),$ $1\leq i \leq k$, in terms of 
$G_{(k-1)j+l}(q), \dots, G_{(k-1)j+k}(q)$ for arbitrary $j\geq 0$, as
\begin{eqnarray}\label{Gi_appendix2}
G_i(q) & = &
{}_ih^{(j)}_1(q) G_{(k-1) j +1}(q) + \dots + {}_ih^{(j)}_k(q) G_{(k-1) j + 
k}(q).
\end{eqnarray}
For each $J\geq 0$, 
$G_{(k-1)J+i}(q)$ can be expressed analogously in such a way that the statements
in \cite{LZ} correspond to the case $J=0$
and such that an analogue of Proposition 2.1 also holds.

Fix an integer $J\geq 0$.
For each $i=1,\dots,k$ and $j\geq J$, 
the use of (2.17) and (2.18) in \cite{LZ} gives
\begin{equation*}
G_{(k-1)J+i}(q)={}_i^Jh^{(j)}_1(q) G_{(k-1) j +1}(q) + 
\cdots + {}_i^Jh^{(j)}_k(q) G_{(k-1) j + k}(q),
\label{eq:GJintermsofhigher} 
\end{equation*}
just as in (2.12).
For each $j\geq J$, the coefficients ${}^J_ih^{(j)}_l(q)$ form a $k \times
k$ matrix ${\bf h}^{(j)}$ of polynomials in $q$ with nonnegative
integral coefficients.  More explicitly, define row vectors
\[
{}^J_i {\bf h}^{(j)} = [{}^J_ih^{(j)}_1(q), \dots, {}^J_ih^{(j)}_k(q)].
\]
For $j=J$ we have
\[
{}^J_i{\bf h}^{(J)} = [0,  \dots, 1, \dots, 0],
\]
with $1$ in the $i^{\text{th}}$ position, so that $^J{\bf h}^{(J)}$ is the
identity matrix.  The ${}^J_i{\bf h}^{(j)}$ satisfy the same set of
recursions with respect to $j$, independently of $i$.  Explicitly:

\begin{proposition}\label{hrecursions}
Let $j \geq J+1$.  With the left subscript ${}_i$ suppressed, we have
\begin{eqnarray*}
{}^Jh^{(j)}_1(q) & = &
{}^Jh^{(j-1)}_1(q) + \cdots + {}^Jh^{(j-1)}_{k-1}(q) +  {}^Jh^{(j-1)}_k(q)
\nonumber \\
{}^Jh^{(j)}_2(q) & = &
({}^Jh^{(j-1)}_1(q) + \cdots + {}^Jh^{(j-1)}_{k-1}(q)) q^j
\nonumber \\
& \cdots &
\nonumber \\
{}^Jh^{(j)}_{k-1}(q) & = &
({}^Jh^{(j-1)}_1(q) + {}^Jh^{(j-1)}_2(q)) q^{(k-2)j}
\nonumber \\
{}^Jh^{(j)}_k(q)
& = &
{}^Jh^{(j-1)}_1(q) q^{(k-1)j}
\nonumber
\end{eqnarray*}
or in general,
\begin{eqnarray}\label{hlj}
{}^Jh^{(j)}_l(q) & = &
({}^Jh^{(j-1)}_1(q) + \cdots +
{}^Jh^{(j-1)}_{k - l +1}(q)) q^{ (l-1)j }, \quad 1 \leq l \leq k.
\end{eqnarray}
In matrix form, this is:
\begin{eqnarray}\label{h=hA}
{{}^J\bf h}^{(j)} = {}^J{\bf h}^{(j-1)}{\bf A}_{(j)},
\end{eqnarray}
with ${}^J\bf h$ the $k \times k$ matrix defined above and with
\begin{eqnarray}\label{Aj_appendix2}
{\bf A}_{(j)} =
\left[ \begin{array}{ccccc}
1 & q^j & q^{2j} & \cdots  & q^{( k-1) j} \\
\vdots & \vdots &  \vdots & \swarrow  & \vdots \\
1 & q^j & q^{2j}  & \cdots & 0 \\
1 & q^j & 0 &  \cdots & 0 \\
1 & 0 & 0 & \cdots  & 0
\end{array} \right].
\end{eqnarray}
In particular,
\begin{eqnarray}\label{hJ=Identity}
{}^J{\bf h}^{(J)}=I
\end{eqnarray}
(the identity matrix) and 
\begin{eqnarray}\label{h=AAA}
{}^J{\bf h}^{(j)} = {\bf A}_{(J+1)}{\bf A}_{(J+2)} \cdots {\bf A}_{(j)}
\end{eqnarray}
for all $j > J$.
\end{proposition}
The proof is completely analogous to the proof of Proposition 2.1 in \cite{LZ}.

\begin{rema}\label{hnonneg_appendix2}
The coefficients of the polynomials ${}^J_i{ h}^{(j)}_l(q)$ for all $j\geq J$,
$i,l=1,\dots,k$ are
nonnegative integers, as can be seen from equation (2.17) in \cite{LZ}
or, equivalently, from equations \eqref{hJ=Identity} and \eqref{h=AAA}
in the matrix formulation above.
\end{rema}

\begin{rema}\label{ehrenpreis}
From \eqref{eq:GJintermsofhigher}, \eqref{hrecursions}
and the Empirical Hypothesis 
it is clear that
\begin{align*}
G_{(k-1)J+i}(q)&=
\left(\lim\limits_{j\rightarrow\infty}{}_i^Jh^{(j)}_1(q)\right)
\left(\lim\limits_{j\rightarrow\infty} G_{(k-1)j+1}(q)\right) + 
\cdots + 
\left(\lim\limits_{j\rightarrow\infty}{}_i^Jh^{(j)}_k(q)\right)
\left(\lim\limits_{j\rightarrow\infty} G_{(k-1)j+k}(q)\right)\\
&=\lim\limits_{j\rightarrow\infty}{}_i^Jh^{(j)}_1(q)
\end{align*}
for all $i=1,\dots,k$ 
(cf. Theorem \ref{thm:GisHinfty} and Remark \ref{rem:essenceofG=h^infty}).
This, combined with Remark \ref{hnonneg_appendix2} immediately implies
that the coefficients of all the $G_l(q)$'s are nonnegative integers.
As mentioned earlier, for $k=2$ and $J=3$ this is precisely the answer
to Ehrenpreis's question about the product sides of the (two)
Rogers-Ramanujan identities. 
\end{rema}

\subsection{$J$-generalization of statements in Proposition 2.2 of \cite{LZ}}

A similar generalization immediately works for Proposition 2.2 in \cite{LZ}.
Proposition 2.2 is the special case $J=0$ of the following proposition:

\begin{proposition}\label{h=partitions}
For each $j \geq J+1$, $i=1,\dots,k$ and $l=1,\dots,k$, 
the polynomial ${}^J_ih^{(j)}_l(q)$ is the generating function for partitions 
with
difference at least 2 at distance $k-1$, 
such that the smallest part is greater than $J$, 
the part $J+1$ appears at most $k-i$ times, 
the largest part is at most $j$,
and the part $j$ appears exactly $l-1$ times.
\end{proposition}

\begin{proof}
Only trivial modification of the proof in [LZ] is needed. 
As in \cite{LZ}, we show that the combinatorial generating functions
described here have the same initial values and recursions as the
polynomials ${}^J_ih^{(j)}_l(q)$.  We say that a partition is of type
$(k-1,J, k-i)$ if it has difference at least 2 at distance $k-1$, 
smallest part larger than $J$, and $J+1$
appearing at most $k-i$ times. 
(Although we use the same notation, these are different from
the partitions discussed in the proof of Proposition \ref{hinterp}.) 
Similarly, \eqref{hlj} corresponds to the
following combinatorial fact: For $j \ge J+2$,
\begin{eqnarray*}
\lefteqn{\text{the number of partitions of $m$ of type $(k-1,J, k-i)$
such that the largest part is at most}}
\nonumber \\
& & \text{$j$ and the part $j$ appears exactly $l-1$ times}
\nonumber \\
& = &
\sum_{p=1}^{k-l+1}
\text{the number of partitions of $m-(l-1)j$ of type $(k-1,J, k-i)$ such
that the}
\nonumber \\
& & \quad
\text{largest part is at most $j-1$ and the part $j-1$ appears
exactly $p-1$ times.}
\nonumber
\end{eqnarray*}

For $J\geq 0$, the initial values
\[
{}^J_i{\bf h}^{(J+1)} = [1, q^{J+1}, q^{2(J+1)}, \cdots, q^{(k-i)(J+1)}, 0, 
\cdots, 0]
\]
also match those of the generating functions.
\end{proof}

\subsection{$J$-generalization of Theorem 2.2 of \cite{LZ}}

Correspondingly, Theorem 2.2 of \cite{LZ} generalizes as follows:
\begin{theo}\label{Gi=difference}
For $i=1,\dots,k$, $G_{(k-1)J+i}(q)$ is the generating function for 
partitions
with difference at least $2$ at distance $k-1$ 
such that smallest part is greater than $J$ and $J+1$ appears
at most $k-i$ times.
\end{theo}
\begin{proof}
This follows immediately from \eqref{eq:GJintermsofhigher}, 
Proposition \ref{h=partitions} and the Empirical Hypothesis.
\end{proof}
\begin{rema}\label{also4.1}
It is interesting to note that Theorem \ref{Gi=difference}
is exactly Theorem 4.1 of \cite{LZ},
which was proved differently in \cite{LZ}.
\end{rema}

Because of this remark, 
it is natural to give a $J$-generalization of Theorem 4.2 of \cite{LZ} here:

\begin{theo}\label{interp}
For $l = 1, \dots, k$, $J\geq 0$ and $j \geq J+1$, the right-hand side of
(\ref{eq:GJintermsofhigher}) expresses the generating function 
$G_{(k-1)J+i}(q)$ as the sum of its
contributions corresponding to the number of times, namely, $0, 1,
\dots, k-1$, that the part $j$ appears in a partition.
\end{theo}

\subsection{$J$-generalization of the matrix interpretation}

One can similarly express all the above $J$-generalizations 
in an illuminating matrix form:

Set
\begin{eqnarray*}
{\bf G}_{(0)} =
\left[ \begin{array}{c}
G_1(q) \\
\vdots \\
G_k(q)
\end{array} \right]
\end{eqnarray*}
and in general,
\begin{eqnarray*}
{\bf G}_{(j)} =
\left[ \begin{array}{c}
G_{(k-1)j+1}(q) \\
\vdots \\
G_{(k-1)j+k}(q)
\end{array} \right]
\end{eqnarray*}
for $j \geq 0$.  Also set
\begin{eqnarray*}
{\bf B}_{(j)} =
\left[ \begin{array}{cccccc}
0           & 0            & \cdots   & 0       & 0        & 1       \\
0           & 0            & \cdots   & 0       & q^{-j}   & -q^{-j} \\
0           & 0            & \cdots   & q^{-2j} & -q^{-2j} & 0       \\
\vdots      & \vdots       & \swarrow & \vdots  & \vdots   & \vdots  \\
0           & q^{-(k-2)j}  & \cdots   & 0       & 0        & 0       \\
q^{-(k-1)j} & -q^{-(k-1)j} & \cdots   & 0       & 0        & 0
\end{array} \right]
\end{eqnarray*}
for $j \geq 1$. Since 
\begin{eqnarray}\label{G=BG}
{\bf G}_{(j)} = {\bf B}_{(j)}{\bf G}_{(j-1)}
\end{eqnarray}
for $j \geq 1$, one has
\[
{\bf G}_{(j)} = {\bf B}_{(j)}{\bf B}_{(j-1)} \cdots {\bf B}_{(J+1)}{\bf
G}_{(J)}
\]
for all $J \geq 0$ and $j\ge J$.  Since
\[
{\bf B}_{(j)} = ({\bf A}_{(j)})^{-1}
\]
(recall (\ref{Aj})) we thus have
\[
{\bf G}_{(J)} = {\bf A}_{(J+1)}{\bf A}_{(J+2)} \cdots {\bf A}_{(j)}{\bf
G}_{(j)}
\]
for $j \geq J$.  Defining ${}^J{\bf h}^{(j)}$ recursively by
\[
{}^J{\bf h}^{(J)} = I,
\]
\[
{}^J{\bf h}^{(j)} = {}^J{\bf h}^{(j-1)}{\bf A}_{(j)}
\]
for $j \geq J+1$, we have that
\[
{}^J{\bf h}^{(j)}{\bf B}_{(j)} = {}^J{\bf h}^{(j-1)},
\]
\[
{}^J{\bf h}^{(j)}={\bf A}_{(J+1)}\cdots{\bf A}_{(j)}
\]
and
\[
{\bf G}_{(J)} = {}^J{\bf h}^{(j)}{\bf G}_{(j)}
\]
for each $j \geq J$.  Thus we have an ``automatic''
reformulation and proof of the $J$-generalized form of 
Proposition \ref{hrecursions}, including a
$J$-generalized form of \eqref{eq:GJintermsofhigher}.

\section{An $(x,q)$-dictionary for the Gordon identities}
\label{app:xqforGRR}

By analogy with what we did in Appendix \ref{app:xqforGGA}, 
we compare Section 2 of \cite{LZ}
with Sections 7.1 and 7.2 of \cite{A3}.
Our main observation is the following (cf. Appendix \ref{app:xqforGGA}): 
After setting up the correct dictionary between
suitable specializations of the formal series $J_{k,i}(a,x,q)$ defined in (7.2.2) of 
\cite{A3} ---
$x$ specialized to successively higher powers of $q$ 
and $a$ specialized to $0$ ---
and the expressions $G_l(q)$ defined in \cite{LZ},
we can match the steps in the proof of Theorem 2.1 of \cite{LZ} with
those in the proof of Lemma 7.1 in \cite{A3}.
We proceed to give this dictionary.

Fix $k\geq 2$ and let $i=0,\dots,k+1$. Let us first recall relevant notation from 
\cite{A3}.
Setting $a\mapsto 0$ in (7.2.1) and (7.2.2) of \cite{A3}, we obtain:
\begin{align}
H_{k,i}(0,x,q)&=\sum\limits_{n\geq 0}(-1)^n 
\frac{x^{kn}q^{kn^2+n-in+{\binom{n}{2}}}(1-x^iq^{2ni})}{(q)_n(xq^n)_{\infty}}\\
J_{k,i}(0,x,q)&=H_{k,i}(0,xq,q)=\sum\limits_{n\geq 0}(-1)^n 
\frac{x^{kn}q^{kn+kn^2+n-in+{\binom{n}{2}}}(1-x^iq^{i+2ni})}
{(q)_n(xq^{n+1})_{\infty}}. \label{def:Jxq}
\end{align}
Note that the expressions are well defined for the ``edge cases''
$i=0$ and $i=k+1$.

For $j\geq 0$, specializing $x\mapsto q^j$ in \eqref{def:Jxq},
we obtain
\begin{align*}
J_{k,i}(0,q^j,q)&=H_{k,i}(0,q^{j+1},q)\\&=
\sum\limits_{n\geq 0}(-1)^n 
\frac{q^{jkn+kn+kn^2+n-in+\binom{n}{2}}(1-q^{ji+i+2ni})}{(q)_n(q^{j+n+1})_{
\infty}}. 
\end{align*}
We rewrite the sum so that the denominator in each summand is $(q)_\infty$:
\begin{align}
J_{k,i}(0,q^j,q)
&=\sum\limits_{n\geq 0}(-1)^n 
\frac{q^{(2k+1)\binom{n}{2} + (k(j+1)+k-i+1)n}(1-q^{n+1})
\cdots(1-q^{n+j})(1-q^{i(2n+j+1)})}{(q)_{\infty}}\label{def:Jq}
\end{align}
and obtain
\begin{align}
J_{k,0}(0,q^j,q)&= 0, \label{Jk0}\\
J_{k,k-i+1}(0,q^j,q)&=\sum\limits_{n\geq 0}(-1)^n 
\frac{q^{(2k+1)\binom{n}{2} + 
(k(j+1)+i)n}(1-q^{n+1})\cdots(1-q^{n+j})
(1-q^{(k-i+1)(2n+j+1)})}{(q)_{\infty}}\label{def:Jqi'}.
\end{align}
It is now clear that for $j\ge 0$ and $i=0,\dots,k+1$,
\begin{align}
G_{(k-1)j+i}(q)&=J_{k,k-i+1}(0,q^j,q)\left(=H_{k,k-i+1}(0,q^{j+1},q)\right)
,\label{dict}
\end{align}
and this is our desired dictionary.

With this setup it is now easy to match the steps in the proof of 
Lemma 7.1 of \cite{A3} and the first half of the
proof of Theorem 2.1 of \cite{LZ}, which is concerned with 
finding closed-form expressions for the series $G_l(q)$. 

We now give details about the ``edge-matching,'' which is the
second assertion in Theorem 2.1 of \cite{LZ}.
Recall Lemma 7.2 of \cite{A3}, with $a\mapsto 0$:
\begin{equation}
J_{k,i}(0,x,q)-J_{k,i-1}(0,x,q)=
(xq)^{i-1}J_{k,k-i+1}(0,xq,q).\label{lem7.2A2} 
\end{equation}
The left-hand side of \eqref{dict} exhibits the (tautological) 
``edge-matching'' 
phenomenon 
$$G_{(k-1)j+k}(q)=G_{(k-1)(j+1)+1}(q).$$
The same phenomenon for the right-hand side is demonstrated
by specializing \eqref{lem7.2A2}
with $x\mapsto q^j,$ $i\mapsto 1$
and noting \eqref{Jk0}:
\begin{equation*}
J_{k,1}(0,q^j,q)-J_{k,0}(0,q^j,q)=(q^jq)^{1-1}J_{k,k-1+1}(0,q^{j+1},q)
\label{eq:Jedge}
\end{equation*}
and hence
\begin{equation*}
J_{k,1}(0,q^j,q)=J_{k,k}(0,q^{j+1},q).
\end{equation*}

\begin{rema}
An analogue of Remark \ref{rem:motivated_to_xq} 
also holds in the context of the Gordon identities,
with the reverse-engineering procedure applied to Theorem 2.1 of \cite{LZ}. 
\end{rema}

\vspace{.3in}

\noindent {\small \sc Department of Mathematics, Rutgers University,
Piscataway, NJ 08854} \\ 
{\em E--mail address}:
\texttt{bud@math.rutgers.edu} \\

\noindent {\small \sc Department of Mathematics, Rutgers University,
Piscataway, NJ 08854} \\
{\em E--mail address}:
\texttt{skanade@math.rutgers.edu} \\

\noindent {\small \sc Department of Mathematics, Rutgers University,
Piscataway, NJ 08854} \\
{\em E--mail address}:
\texttt{lepowsky@math.rutgers.edu} \\

\noindent {\small \sc Department of Mathematics, Rutgers University,
Piscataway, NJ 08854} \\
{\em E--mail address}:
\texttt{rhmcrae@math.rutgers.edu} \\
\noindent Current address:\\
\noindent{\small \sc 
Beijing International Center for Mathematical Research, Peking University, Beijing, China 100084}\\
{\em E--mail address}:
\texttt{robertmacrae@math.pku.edu.cn} \\

\noindent {\small \sc Department of Mathematics, Rutgers University,
Piscataway, NJ 08854} \\
{\em E--mail address}:
\texttt{fq15@math.rutgers.edu} \\

\noindent {\small \sc Department of Mathematics, Rutgers University,
Piscataway, NJ 08854} \\
{\em E--mail address}:
\texttt{russell2@math.rutgers.edu} \\

\noindent {\small \sc Department of Mathematics, Rutgers University,
Piscataway, NJ 08854} \\
{\em E--mail address}:
\texttt{sadowski@math.rutgers.edu} \\
\noindent Current address:\\
\noindent{\small \sc 
Ursinus College, Collegeville, PA 19426}\\
{\em E--mail address}:
\texttt{csadowski@ursinus.edu} \\

\end{document}